\documentclass[11pt,twoside,a4paper]{article}

\usepackage{amssymb}
\usepackage{amsmath}
\usepackage{amsthm, color}

\allowdisplaybreaks

\usepackage{amsfonts}
\usepackage{bbm}
\usepackage{verbatim}
\usepackage{url}
\usepackage[T1]{fontenc}
\usepackage{mathrsfs}

\def\binfty {\dot{B}_{\lambda,\infty}^{\beta,\infty}(\mathbb R_+)}

\def\rr{{\mathbb R}}

\newcommand{\GG}{ {\mathop \cg \limits^{    \circ}}_1}
\def\rrp{{{\mathbb\rr}_+}}
\def\lttz{{L^t(\rrp,\, dm_\lz)}}
\newcommand{\R}{\mathbb{R}}
\def\cg{{\mathcal G}}

\newcommand{\GGone}{ {\mathop \cg \limits^{    \circ}}_1}
\def\fz{\infty}
\def\az{\alpha}
\def\supp{{\mathop\mathrm{\,supp\,}}}
\def\lz{\lambda}

\def\pa{\partial}
\def\wz{\widetilde}
\def\dmzy{{\,dm_\lz(y)}}
\def\inzf{{\int_0^\fz}}
\def\pplz{t\partial_tP_t^{[\lambda]}}
\def\ls{\lesssim}
\def\gs{\gtrsim}
\def\tbz{{\triangle_\lz}}
\def\dmz{{dm_\lz}}

\def\riz{{R_{\Delta_\lz}}}
\def\qlz{{Q^{[\lz]}_t}}
\def\plz{{P^{[\lz]}_t}}
\def\ltz{{L^2(\rr_+,\, dm_\lz)}}
\def\loz{{L^1(\rr_+,\, dm_\lz)}}
\def\lpz{{L^p(\rr_+,\, dm_\lz)}}
\def\hoz{{H^1(\rr_+,\, dm_\lz)}}
\def\dint{\displaystyle\int}
\def\r{\right}
\def\lf{\left}
\newtheorem{thm}{Theorem}[section]
\newtheorem{lem}[thm]{Lemma}
\newtheorem{prop}[thm]{Proposition}
\newtheorem{rem}[thm]{Remark}
\newtheorem{defn}[thm]{Definition}

\numberwithin{equation}{section}

\begin{document}

\arraycolsep=1pt

\title{\Large\bf  The Riesz Transform and Fractional Integral Operators\\ in the Bessel Setting}
\author{Jorge J. Betancor, Xuan Thinh Duong, Ming-Yi Lee, Ji Li and Brett D. Wick}

\date{}
\maketitle


\begin{center}
\begin{minipage}{13.5cm}\small

{\noindent  {\bf Abstract:}\  Fix $\lambda>0$. Consider  
the Bessel operator
$\triangle_\lambda:=-\frac{d^2}{dx^2}-\frac{2\lambda}{x}
\frac d{dx}$ on $\mathbb{R_+}$, where $\mathbb{R_+}:=(0,\infty)$ and $dm_\lambda:=x^{2\lambda}dx$ with $dx$ the Lebesgue measure.
We provide a deeper study of the Bessel Riesz transform and fractional integral operator via the related Besov and Triebel--Lizorkin spaces associated with $\triangle_\lambda$.
Moreover, we investigate some possible characterization of the commutator of fractional integral operator, which was missing in the literature of the Bessel setting. 

}

\end{minipage}
\end{center}

%
%
\bigskip


\medskip


\section{Introduction and Statement of Main Results}\label{s1}

\subsection{Background}
In 1965, Muckenhoupt and Stein in \cite{ms} introduced the harmonic function theory associated with the Bessel operator $\tbz$, defined by,
\begin{equation*}
\tbz :=-\frac{d^2}{dx^2}-\frac{2\lz}{x}\frac{d}{dx},\quad \lz>0.
\end{equation*}
They developed a theory in the setting of
$\tbz$ which parallels the classical one associated to the usual Laplacian $\triangle$.  Results on the $\lpz$-boundedness of conjugate
functions and fractional integrals associated with $\tbz$ were
obtained, where $p\in[1, \fz)$, $\rr_+:=(0, \fz)$ and $\dmz(x):= x^{2\lz}\,dx$.

The related elliptic partial differential equation is the following ``singular Laplace equation''
\begin{equation}\label{bessel laplace equation}
\triangle_{t,\,x} (u) :=-\pa_{t}^2u - \pa_{x}^2u-\frac{2\lz}{x}\pa_{x}u=0
\end{equation}
studied by Weinstein \cite{w48}, and Huber in higher dimension in \cite{Hu}, where they considered the generalised axially symmetric potentials, and obtained the properties of the solutions of this equation, such as the extension, the uniqueness theorem, and the boundary value problem for certain domains.

Since then, the Bessel context has been extensively studied;
see, for example, \cite{ak,bcfr,bfbmt,bfs,bhnv,k78,v08,yy} and the references therein.
Comparing to the standard Laplacian on Euclidean space, the key difference is that there is no Fourier transform in the Bessel setting. Thus, some fundamental results on $\mathbb R^n$ were only recently established in the Bessel setting via different approaches that avoided arguments involving the Fourier transform.

For example, in the Euclidean setting, the characterization of the boundedness of the commutator of the Riesz transforms via the BMO space was first established by Coifman--Rochberg--Weiss \cite{crw}.  In his setting, the first proof of the lower bound of the commutator utilized a Fourier expansion. A parallel version of \cite{crw} was established in the Bessel setting in \cite{dgklwy} and \cite{DLWY} via a real analysis method, and later on the characterization of  compactness of commutator in the Bessel setting was found in \cite{DLMWY}.

However, there are several other aspects of the commutator $[b,R_{\tbz }]$ of the Bessel Riesz transform that remain open. 

\noindent {\bf Question 1}:  The connection between the commutator of the Bessel Riesz transform and the Besov space or Triebel--Lizorkin space in the Bessel setting ?

Note that in $\mathbb R^n$ in the Euclidean setting a connection was established by M. Paluszy\'nski \cite{P}, where the approach depends heavily on the Fourier transform.

\noindent {\bf Question 2}:  The characterization of the endpoint boundedness of the commutator of the Bessel Riesz transform? That is, to investigate the necessary and sufficient conditions on the symbol $b$ such that $[b,R_{\tbz }]$ is bounded from the Bessel Hardy space to $L^1$, and from $L^\infty$ to the Bessel BMO space.

It turns out that a simple duplication of the classical result on $\mathbb R^n$ (\cite{Pe}), especially from $L^\infty$ to the Bessel BMO space, is not the suitable one, and hence this remains open till now.

\noindent {\bf Question 3}:  The characterization of the boundedness of the commutator $[b,\tbz ^{-\alpha/2}]$ of the fractional operator  $\tbz ^{-\alpha/2}$?

It turns out to be a more difficult question in comparison to the Euclidean setting. The argument is different from the classical result on $\mathbb R^n$, due to the intrinsic  structure of the Bessel Laplacian.

\bigskip
The aim of this paper is to address these questions in the Bessel setting.


\subsection{Main results}

We will regard $\mathbb{R}_{+}$ as a space of homogeneous type, in the sense of Coifman and Weiss, \cite{cw77}, with Euclidean metric (i.e., the absolute value $|\cdot|$) and measure $m_\lambda$. Specifically, for any $x\in\mathbb{R}_+$ and $r>0$, a set $I(x,r):=B(x,r)\cap \mathbb{R}_+$ is considered as a ball in $\mathbb{R}_{+}$, where $B(x,r)$ is a Euclidean ball with centre $x$ and radius $r$. It can be deduced from \cite{DLMWY} that $m_\lambda$ on $\mathbb{R}_{+}$ satisfies a doubling condition stated as follow:  for every $I\subset \mathbb{R}_{+}$,
\begin{align}
\label{doub}\min\{2^{2},2^{2\lambda+1}\}m_\lambda(I)\leq m_\lambda(2I)\leq 2^{2\lambda+1}m_\lambda(I).
\end{align}
Set $\mathbf Q=2\lambda+1$, which is the upper dimension of the measure $m_\lambda$.
\medskip

Next, consider $(\mathbb{R}_{+}, |\cdot|, m_\lambda)$ as a space of homogeneous type. We let $\{S_k\}_{k\in\mathbb Z}$ be a family of operators on $\mathbb R_+$ which is an approximation to the identity and set $D_k = S_k-S_{k+1}$. Full details are provided in Section 2.1. Moreover, for $0<\beta_0,\gamma_0<1$, we let $(\GG(\beta_0,\gamma_0))'$ denote the space of distributions on $(\mathbb{R}_{+}, |\cdot|, m_\lambda)$, with full details in Section 2.2.

We first consider the Besov space  $\binfty$ associated with $\tbz$ and then study the commutator with the Riesz transform and fractional integral operators in the Bessel setting.
\begin{defn}\label{def Besov}
Let $0<\beta_0,\gamma_0<1$. For $\beta\in (0,1)$, the Besov space $\binfty$  is the set of distributions $f$ in $(\GG(\beta_0,\gamma_0))'$ such that 
$\|f\|_{\binfty}:=\sup\limits_{\substack{k\in\mathbb Z, \ x \in \mathbb R_+}}2^{\beta k}|D_kf(x)|<\infty.
   $
\end{defn}    

Next, we provide the definition of the Triebel--Lizorkin spaces associated with $\tbz $ via the 
operator $\{D_k\}_{k\in\mathbb Z}$ from the
approximation to identity $\{S_k\}_{k\in\mathbb Z}$.


\begin{defn}\rm
Let $0<\beta_0,\gamma_0<1$. For $|\alpha|<1$, $1< p<\infty$ and $1\leq q\leq\infty$, the Triebel--Lizorkin spaces associated with $\tbz$ are defined as follows.  Set
$$\dot{\mathcal F}^{\alpha,q}_{\lambda,p}( \mathbb R_+) :=\big\{f\in (\GG(\beta_0,\gamma_0))' : \|f\|_{\dot{\mathcal F}^{\alpha,q}_{\lambda,p}} :<\infty\big\},$$
where
\begin{align*}
\|f\|_{\dot{\mathcal F}^{\alpha,q}_{\lambda,p}}
	=\begin{cases} \displaystyle \Big\|\Big\{\ \sum_{k\in \mathbb Z}  \big(2^{k\alpha}|D_k(f)|\big)^q \Big\}^{1/q}\Big\|_{\lpz}\quad & \text{if}\ \ 1\le q<\infty, \\[9pt]
	\displaystyle \Big\|\sup_{k\in \mathbb Z} 2^{k\alpha}|D_k(f)|\Big\|_{\lpz} &\text{if}\ \ q=\infty. \end{cases}
	\end{align*}
	
\end{defn}
It is important to note that the definition of $\dot{\mathcal F}^{\alpha,q}_{\lambda,p}( \mathbb R_+) $ is independent of the choice of $\{D_k\}_{k\in\mathbb Z}$. When $q=\infty$, this is justified by our Theorem \ref{prop 3.1}. When $1\leq q<\infty$, this can be seen from the Plancherel--Polya type inequality for Triebel--Lizorkin spaces on space of homogeneous type from  Han, M\"uller and Yang \cite{hmy06,hmy08}.

This leads to our first main result:  characterisation for the Besov space  $\binfty$ via the commutator $[b,R_{\tbz }]$, which is defined as 
$[b,R_{\tbz }](f) = bR_{\tbz }(f)- R_{\tbz }(bf)$.
 \begin{thm}\label{thm0} Let $1<p<\infty, 0<\beta<1$. The following conditions are equivalent:
\begin{itemize}
\item[{\rm(a)}] $b\in \binfty$;
\item[{\rm(b)}] $[b,R_{\tbz }]$ is  bounded from $\lpz$ to $\dot F_{\lambda ,p}^{\beta,\infty}(\mathbb R_+)$.
\end{itemize}
\end{thm}
This resolves  {\bf Question 1} and provides a different method (bypassing the use of the Fourier transform) from the result of M. Paluszy\'nski \cite{P} in $\mathbb R^n$.

\medskip
We now turn to the second main result:  characterization of endpoint boundedness for 
$[b, \riz]$. Recall that the BMO space is defined as follows.

\begin{defn}
A function $f\in L^1_{\rm loc}(\rrp,dm_\lambda)$ belongs to
the {\it space} ${\rm BMO}_{\lambda}(\mathbb R_+)$ if
\begin{equation*}
\sup_{x,\,r\in(0,\,\fz)}\frac1{m_\lz(I(x,\,r))}\dint_{I(x,\,r)}
\lf|f(y)-\frac1{m_\lz(I(x, r))}\dint_{I(x,\,r)}f(z)\,dm_\lz(z)\r|\,dm_\lz(y)<\fz.
\end{equation*}
\end{defn}

\begin{thm}\label{thm bmo4.1}
Let $\lambda>0$. Assume that $b\in {\rm BMO}_{\lambda}(\mathbb R_+)$. Then, the following assertions are equivalent:
\begin{itemize}
\item[{\rm (a)}] $[b, \riz]$ is bounded from $L^\infty_c(\mathbb R_+, dm_\lambda)$ to ${\rm BMO}_{\lambda}(\mathbb R_+)$;
\item[{\rm (b)}] $\sup\limits_{I=(x_0-r_0,x_0+r_0) \atop 0<4r_0<x_0} \log\big(\frac{x_0}{r_0}\big)\frac1{m_\lambda(I)}
                               \int_I |b(z)-b_I|dm_\lambda(z)<\infty.$
\end{itemize}
\end{thm}

\begin{defn}[\cite{bdt}]\label{d-atomic H1}
 A function $a$ is called a $(1, 2)_{\tbz }$-atom
if there exists an open bounded interval $I\subset \mathbb{R}_+$
such that $\supp(a)\subset I$, $\|a\|_{L^2(\mathbb R_+,dm_\lambda)}\le[m_\lz(I)]^{-1/2}$
and $\int_0^\fz a(x)\,\dmz(x)=0.$
\end{defn}

\begin{thm}\label{thm bmo4.2}
Let $\lambda>0$. Assume that $b\in {\rm BMO}_{\lambda}(\mathbb R_+)$. Then, $[b, \riz]$ can be extended from 
span$\{(1, 2)_{\tbz }$-atoms$\}$ to $H^1(\mathbb R_+, dm_\lambda)$ as a bounded operator from $H^1(\mathbb R_+, m_\lambda)$ into 
$L^1(\mathbb R_+, m_\lambda)$ if, and only if, $b$ is constant.
\end{thm}

Theorems \ref{thm bmo4.1} and \ref{thm bmo4.2} give an affirmative answer to {\bf Question 2}. Moreover, it is worth pointing out that comparing the result to the classical result on $\mathbb R^n$ \cite{Pe} (the commutator of the Riesz transforms are bounded from $L^\infty$ to BMO if and only if $b$ is a constant),  the condition (b) in Theorem \ref{thm bmo4.1} is new and reveals  the intrinsic structure of the Bessel setting. 



We now turn to our third main result.
 Recall the fractional operator $\tbz ^{-\alpha/2}$ with $0<\alpha<{\bf Q}$, defined by
$$ \tbz ^{-\alpha/2}f(x) = {1\over \Gamma(\alpha/2)} \int_0^\infty { e^{t\tbz }}(f)(x) {dt\over t^{1-\alpha/2}}. $$

We investigate the    commutator $[b,\tbz ^{-\alpha/2}](f)(x):= b(x)\tbz ^{-\alpha/2}(f)(x) - \tbz ^{-\alpha/2}(bf)(x) $ as follows.
    


\begin{thm}\label{thm 1.7}
Let $1<p<q<\infty, {0<\alpha<1}$, and $0<\beta<1$. 
If $b\in \binfty$ then 
$[b,\tbz ^{-\alpha/2}]$ is  bounded from $\lpz$ to $\dot F_{\lambda ,q}^{\beta,\infty}(\mathbb R_+)$ for  $\frac1p-\frac1q=\frac{\alpha}{\bf Q}$.
%
\end{thm}

We point out that the reverse of this result is still open.
Moreover, the characterization the $L^p-L^q$ boundedness of $[b,\tbz ^{-\alpha/2}]$ is still open, which may or may not correspond to the existing Bessel BMO space.

%

This paper is organised as follows. In Section 2, we provide the necessary preliminaries in the Bessel setting. In Section 3 we provide a characterization of the Besov space in the Bessel setting and provide the proof of Theorem \ref{thm0}. In Section 4 we give the proof of Theorems \ref{thm bmo4.1} and \ref{thm bmo4.2}.  In the last section we prove Theorem \ref{thm 1.7}.


\section{Preliminaries}

\subsection{Approximation to the Identity}

We first recall the definition of approximation to the identity.
\begin{defn}\label{def-ati}
We say that a family of operators $\{S_k\}_{k\in\mathbb Z}$ on $\mathbb R_+$ is an approximation to the identity if $\displaystyle\lim_{k\to+\infty} S_k=Id$, $\displaystyle\lim_{k\to -\infty} S_k=0$
and moreover, the kernel $S_k(x,y)$ of $S_k$ satisfies the following condition:  for $\beta,\gamma\in (0,1]$,
\begin{itemize}
\item[${\rm (A_i)}$] for any $x,\,y \in\mathbb R_+$ and $k \in\mathbb Z$,
\begin{equation*}
|S_k(x, y)|\ls \frac1{m_\lz(I(x, 2^{-k}))+m_\lz(I(y,2^{-k}))+m_\lz(I(x, |x-y|))}\bigg(\frac {2^{-k}}{|x-y|+2^{-k}}\bigg)^\gamma;
\end{equation*}
  \item [${\rm (A_{ii})}$] for any $x,y,\wz y\in\mathbb R_+$  and $k \in\mathbb Z$ with $|y-\wz y|\le (2^{-k}+|x-y|)/2$,
\begin{align*}
&|S_k(x, y)-S_k(x,\wz y)|+ |S_k(y,x)-S_k(\wz y,x)|\\
&\ls \frac1{m_\lz(I(x, 2^{-k}))+m_\lz(I(y,2^{-k}))+m_\lz(I(x, |x-y|))}\bigg(\frac {|y-\wz y|}{|x-y|+2^{-k}}\bigg)^{\beta}\bigg(\frac {2^{-k}}{|x-y|+2^{-k}}\bigg)^\gamma;
\end{align*}
  \item [${\rm (A_{iii})}$]for any $x,\,y,\,\wz x,\,\wz y\in\mathbb R_+$  and $k\in\mathbb Z$ with $|x-\wz x|$, $|y-\wz y|\le (2^{-k}+|x-y|)/3$,
 \begin{align*}
&|S_k(x, y)-S_k(x,\wz y)|+|S_k(\wz x,y)-S_k(\wz x,\wz y)|\\
&\ls \frac1{m_\lz(I(x, 2^{-k}))+m_\lz(I(y,2^{-k}))+m_\lz(I(x, |x-y|))}\\
&\quad\times \bigg(\frac {|x-\wz x|}{|x-y|+2^{-k}}\bigg)^{\beta}\bigg(\frac {|y-\wz y|}{|x-y|+2^{-k}}\bigg)^{\beta}\bigg(\frac {2^{-k}}{|x-y|+2^{-k}}\bigg)^\gamma;
\end{align*}
  \item [${\rm (A_{iv})}$] for any $x\in\mathbb R_+$  and $k \in\mathbb Z$,
\begin{equation*}
\inzf S_k(x, y)\dmzy=\inzf S_k(y,x)\dmzy=1.
\end{equation*}
\end{itemize}
\end{defn}
One of the constructions of an approximation to the identity is due to
Coifman, see \cite{DJS}. We set $D_k:=S_k-S_{k+1}$. And it is obvious that the kernel $D_k(x,y)$
of $D_k$ satisfies ${\rm (A_i)}$, ${\rm (A_{ii}) }$ and ${\rm (A_{iii})}$ with $S_k(x,y)$ replaced by $D_k(x,y)$, and
\begin{equation*}
\inzf D_k(x, y)\dmzy=\inzf D_k(y,x)\dmzy=0
\end{equation*}
for any $x\in\mathbb R_+$  and $k \in\mathbb Z$.

\subsection{Test Function Space and Distributions}
We now
recall  the test function spaces and distribution spaces defined by Han, M\"uller and Yang \cite{hmy06,hmy08}.

\begin{defn}[\cite{hmy06}]\label{def-of-test-func-space}
Consider the space $(\mathbb R_+,|\cdot|,\dmz)$. Let $0<\gamma, \beta\leq 1$ and $r>0.$
A function $f$ defined on $\mathbb R_+$ is said to be a test function of type
$(x_0,r,\beta,\gamma)$ centered at $x_0\in \mathbb R_+$ if $f$ satisfies the
following conditions:
\begin{itemize}
\item[\rm(i)]
$|f(x)|\leq C \frac{\displaystyle 1}{\displaystyle
m_\lz(I(x_0,r))+m_\lz(I(x_0,|x_0-x|))} \Big(\frac{\displaystyle r}{\displaystyle
r+|x-x_0|}\Big)^{\gamma}$;
\item[\rm(ii)]
$|f(x)-f(y)|\leq C \Big(\frac{\displaystyle
|x-y|}{\displaystyle r+|x-x_0|}\Big)^{\beta} \frac{\displaystyle 1}{\displaystyle
m_\lz(I(x_0,r))+m_\lz(I(x_0,|x_0-x|))} \Big(\frac{\displaystyle r}{\displaystyle
r+|x-x_0|}\Big)^{\gamma}$ for all $x,y\in \mathbb R_+$ with
$|x-y|\le{\frac{1}{2}}(r+|x-x_0|).$
\end{itemize}
If $f$ is a test function of type $(x_0,r,\beta,\gamma)$, we write
$f\in \cg(x_0,r,\beta,\gamma)$ and the norm of $f\in
\cg(x_0,r,\beta,\gamma)$ is defined by
$\|f\|_{\cg(x_0,\,r,\,\beta,\,\gamma)}:=\inf\{C>0: {\rm (i)\ and\ (ii)\ hold}\}.$
\end{defn}

Now for any fixed $x_0\in \mathbb R_+$, we denote
$\cg(\beta,\gamma):=\cg(x_0,1,\beta,\gamma)$ and by $\cg_0(\beta,\gamma)$ the collection of all test functions in $\cg(\beta,\gamma)$ with $\int_{\mathbb R_+} f(x) \dmz(x)=0.$ Note that
$\cg(x_1,r,\beta,\gamma)=\cg(\beta,\gamma)$ with equivalent norms for all
$x_1\in \mathbb R_+$ and $r>0$ and that
$\cg(\beta,\gamma)$ is a Banach space with respect to the norm in
$\cg(\beta,\gamma)$.

Let $\GG(\beta,\gamma)$ be the
completion of the space $\cg_0(1,1)$ in
the norm of $\cg(\beta,\gamma)$ when $0<\beta,\gamma<1$. If $f\in \GG(\beta,\gamma)$, we then define
$\|f\|_{\GG(\beta,\gamma)}:=\|f\|_{\cg(\beta,\gamma)}$. $(\GG(\beta,\gamma))'$, the distribution space, is defined to be the set of all
linear functionals $L$ from $\GG(\beta,\gamma)$ to $\mathbb{C}$ with
the property that there exists $C\geq0$ such that for all $f\in
\GG(\beta,\gamma)$,
$|L(f)|\leq C\|f\|_{\GG(\beta,\gamma)}.$

\subsection{Poisson and Conjugate Poisson Kernels, and the Riesz Transform in the Bessel Setting}\label{s2}



Let $P^{[\lz]}_t:= e^{-t\sqrt{\Delta_\lz}}$ be the Poisson semigroup associated with $\Delta_\lz$.
We recall the following properties from \cite{bdt}.

\begin{lem}\label{l-seimgroup prop}
Then $\{P^{[\lz]}_t\}_{t>0}$ is
a symmetric diffusion semigroup satisfying that $P^{[\lz]}_tP^{[\lz]}_s=P^{[\lz]}_sP^{[\lz]}_t$ for any $t,\,s\in(0, \fz)$,
the $C_0$ property, $\lim_{t\to0}P^{[\lz]}_tf=f$ in $\ltz$ and
\begin{itemize}
  \item [${\rm (S_i)}$]$\|P^{[\lz]}_tf\|_{\lpz}\le \|f\|_{\lpz}$ for all $p\in[1, \fz]$ and $t\in(0,\,\fz)$;
  \item[${\rm (S_{ii})}$] $P^{[\lz]}_t f\ge0$ for all $f\ge0$ and $t\in(0, \fz)$;
  \item [${\rm (S_{iii})}$] $P^{[\lz]}_t(1)=1$ for all $t\in (0,\,\fz)$.
\end{itemize}
\end{lem}

Next we recall the definitions of the Poisson kernel and conjugate Poisson kernel, see \cite{bdt}.
For any $t,\,x,\,y\in(0, \fz)$,
\begin{equation*}
P^{[\lz]}_tf(x):=\inzf P^{[\lz]}_t(x,y)f(y)\,dm_\lambda(y),
\end{equation*}
where
\begin{equation*}
P^{[\lz]}_t(x,y)=\frac{2\lz t}{\pi}\int_0^\pi\frac{(\sin\theta)^{2\lz-1}}{(x^2+y^2+t^2-2xy\cos\theta)^{\lz+1}}\,d\theta.
\end{equation*}

If $f\in\lpz$, $p\in[1, \fz)$, the $\Delta_\lz$-conjugate of $f$ is defined by setting,
for any $t,\,x,\,y\in(0, \fz)$,
\begin{equation}\label{conj poi defn}
\qlz(f)(x):=\int_0^\fz\qlz(x, y)f(y)\, dm_\lz(y),
\end{equation}
where
\begin{equation*}
\begin{array}[b]{cl}
\qlz(x, y)&:=-\dfrac{2\lz}{\pi}\dint_0^\pi\dfrac{(x-y\cos\theta)(\sin\theta)^{2\lz-1}}
{(x^2+y^2+t^2-2xy\cos\theta)^{\lz+1}}\,d\theta;
\end{array}
\end{equation*}
see \cite[p.\,84]{ms}. We point out that the boundary value function $\lim_{t\to0}\qlz(f)(x)$
exists for almost every $x\in (0, \fz)$ (see \cite[p.\,84]{ms}),
which is defined to be the Riesz transform $\riz(f)$, i.e.,
\begin{align}\label{riz}
\riz(f)(x):=\lim_{t\to0}\qlz(f)(x) = \int_{\R_+}  -\dfrac{2\lz}{\pi}\dint_0^\pi\dfrac{(x-y\cos\theta)(\sin\theta)^{2\lz-1}}
{(x^2+y^2-2xy\cos\theta)^{\lz+1}}\,d\theta \ f(y) \dmz(y).
\end{align}

\begin{prop}[\cite{yy}]\label{p-aoti}
For any fixed $t$ and $x\in\R_+$,
$ \plz(x, \cdot)$, $ \qlz(x, \cdot)$, $ \pplz(x, \cdot)$ and $ t\pa_y\plz(x, \cdot)$
as  functions of $x$ are in $\GGone(\beta,\gamma)$ for all $\beta,\gamma\in (0,1]$.  Symmetrically,  for any
fixed $t$ and $y\in\R_+$, $\pplz(\cdot, y)$  and $ t\pa_y\plz(\cdot, y)$ are in $\GGone(\beta,\gamma)$ for all $\beta,\gamma\in (0,1]$.
\end{prop}

We note that, as indicated in \cite{bdt}, see also \cite{bfs}, this Riesz transform $\riz$ is a Calder\'on--Zygmund operator.
In order to prove main results we need the following properties of the kernel $\riz(x,y), x,y\in \mathbb R_+$, of the Riesz
transform $\riz$ that can be found in \cite{BCN, DLWY}.

\begin{prop}\label{t:RieszCZ}
Let $\lambda>0$. There exists constants $K_0, K_1>1$ such that
\begin{itemize}
  \item [\rm(a)] $|\riz(x,y)|\lesssim \begin{cases} \frac1{x^{2\lambda+1}}, &0<y<K_0^{-1}x,\\
                  \frac1{y^{2\lambda+1}}, &0<K_0 x<y;  \end{cases}$
  \item [\rm (b)] For every $x,y\in\mathbb{R}_+$ with $x\not=y$,
 $$  |\riz(x,y)|\ls \frac1{m_\lz(I(x, |x-y|))} ;$$  
 \item [\rm (c)] For every $x,\,y,\, u\in \mathbb{R}_+$ with $|x-y|>\frac32|x-u|$, 
$$|\riz(x, y)-\riz(u,y)|+ |\riz(y, x)-\riz(y,u)|\ls \frac{|x-u|}{|x-y|}\frac1{m_\lz(I(x, |x-y|))};$$
\item [\rm (d)]  $\riz(x,y)\gtrsim \frac1{(xy)^\lambda(y-x)}, 0<x<y<K_0x;$
\item [\rm (e)]  $\riz(x,y)\lesssim -\frac1{(xy)^\lambda(x-y)}, 0<K_0^{-1}x<y<x;$
\item [\rm (f)]    $\riz(x,y)\gtrsim \frac x{y^{2\lambda+2}}, 0<K_1x<y;$
\item [\rm (g)]  $\riz(x,y)\lesssim -\frac1{x^{2\lambda+1}}, 0<y<K_1^{-1}x.$
\end{itemize}
\end{prop}


\subsection{The Hardy Space in the Bessel Setting}

We now recall the atomic characterization of the Hardy spaces $\hoz$ in \cite{bdt}.
From \cite{bdt}, the Hardy space $\hoz$ can be characterized via an atomic decomposition. That is,
an $\loz$ function $f\in \hoz$ if and only if
$$f=\sum_{j=1}^\fz\az_j a_j\quad {\rm in}\quad  \loz,$$ where for every $j$,
$a_j$ is a $(1, 2)_{\tbz }$-atom and $\az_j\in\rr$ satisfies that
$\sum_{j=1}^\fz|\az_j|<\fz$. Moreover,
$\|f\|_\hoz\sim \inf\lf\{\sum_{j=1}^\fz|\az_j|\r\},$
where the infimum is taken over all the decompositions of $f$ as above.

We also note that $\hoz$  can be
characterized in terms of the radial
maximal function associated with the Hankel convolution of a class
of functions, including the Poisson semigroup and the heat
semigroup as special cases.  It is also  proved in \cite{bdt} that $H^1(\mathbb{R}_+, \dmz)$ is the one associated with
the space of homogeneous type $(\mathbb{R}_+, \rho, dm_\lz)$ defined by Coifman and Weiss in \cite{cw77}.


\section{The Lipschitz Space $\Lambda^\beta$, $0<\beta<1$, via the Commutator of the Riesz Transform: Proof of Theorem \ref{thm0}}

To prove Theorem \ref{thm0}, we need the characterization of the Besov space as follows.
We recall the standard Lipschitz space on $\mathbb R_+$, denoted by $\Lambda^\beta$.
\begin{defn}\label{def Lip}
For $\beta\in (0,1)$, the Lipschitz space $\Lambda^\beta$  is the set of functions defined on $\mathbb R_+$ such that 
 $$\|f\|_{\Lambda^\beta}:=\sup_{x,y\in\mathbb R_+:\ x\ne y} \frac{|f(x)-f(y)|}{|x-y|^\beta}<\infty.$$
\end{defn}    


We note that 
$\Lambda^\beta$ is equivalent to the Besov space  $\binfty$ associated with ${\tbz }$.
\begin{thm}\label{lip}
For $\beta\in (0,1)$ and $q\in [1, \infty)$, the following statements are equivalent.
\begin{itemize}
\item[{\rm (i)}] $f\in \Lambda^\beta;$
\item[{\rm (ii)}] $\displaystyle\sup_{I(x_0,r)}\frac 1{r^{\beta}}\bigg(\frac1{m_\lambda(I)}\int_I |f(x)-f_I|^q dm_\lambda(x)\bigg)^{\frac1q}<\infty$; 
\item[{\rm (iii)}] $f$ belongs to Besov space $\binfty$.
\end{itemize}
\end{thm}
We note that this result is not completely new. The proof of (i)$\Leftrightarrow$(iii) was hidden in the previous related result on spaces of homogeneous type by Deng and Han \cite[Page 98, Lemma 4.3]{DH}, where they assumed the measure of the ball is equivalent to the radius of the ball. However, this assumption is not essential. And the equivalence (i)$\Leftrightarrow$(ii) follows from \cite[Theorem 2,4]{N}.

\begin{rem}
Theorem \ref{lip} shows that the space $\binfty$ does not depend on  $\lambda$.
\end{rem}

We also need the difference characterization of
Triebel--Lizorkin spaces.

\begin{thm}\label{prop 3.1}
For $0<\beta<1$ and $1<p<\infty$, we have
$$\|f\|_{{{\dot F}^{\beta,\infty}}_{\lambda ,p}}\sim 
\bigg\|\sup_{k\in \mathbb Z} \frac 1{m_\lambda(I(\cdot, 2^{-k}))2^{-k \beta}}\int_{I(\cdot, 2^{-k})} |f(y)-f_{I(\cdot,2^{-k})}|dm_\lambda(y) \bigg\|_\lpz.$$ 
\end{thm}

To prove this theorem, we need the difference characterization of
Triebel--Lizorkin spaces on spaces of homogeneous type see \cite[Propositions 4.1 and 4.6]{WHYY}.

\begin{prop}\label{defference}
For $0<\beta<1<p<\infty$, we have
$$\|f\|_{{\dot F}^{\beta,\infty}_{\lambda ,p}}\sim \bigg\|\sup_{k\in \mathbb Z} \frac 1{m_\lambda(I(\cdot, 2^{-k}))
{2^{-k\beta}}}\int_{I(\cdot, 2^{-k})} |f(\cdot)-f(y)|dm_\lambda(y)\bigg\|_\lpz.$$ 
\end{prop}

\begin{proof}[Proof of Theorem \ref{prop 3.1}]
Fix $x\in \mathbb R_+$ and $k\in \mathbb Z$.
\begin{align*}
 &\frac 1{m_\lambda(I(x, 2^{-k})){ 2^{-k\beta}}}\int_{I(x, 2^{-k})} |f(y)-f_{I(x,2^{-k})}|dm_\lambda(y)\\
 &\quad \le \frac1{m_\lambda(I(x, 2^{-k})){ 2^{-k\beta}}}\int_{I(x, 2^{-k})} |f(y)-f(x)|dm_\lambda(y) \\
      &\qquad\qquad + \frac1{m_\lambda(I(x, 2^{-k})){ 2^{-k\beta}}}\int_{I(x, 2^{-k})} |f(x)-f_{I(x, 2^{-k})}|dm_\lambda(y) \\
  &\quad \le \frac1{m_\lambda(I(x, 2^{-k})){ 2^{-k\beta}}}\int_{I(x, 2^{-k})} |f(y)-f(x)|dm_\lambda(y) \\
      &\qquad\qquad + \frac1{m_\lambda(I(x, 2^{-k})){ 2^{-k\beta}}} \frac1{m_\lambda(I(x, 2^{-k}))}\int_{I(x, 2^{-k})}\int_{I(x, 2^{-k})}  |f(x)-f(z)|dm_\lambda(z)dm_\lambda(y) \\
   &\quad \le \frac2{m_\lambda(I(x, 2^{-k})){ 2^{-k\beta}}}\int_{I(x, 2^{-k})} |f(y)-f(x)|dm_\lambda(y).
  \end{align*}
Hence, by using Proposition \ref{defference}, 
\begin{align*}
& \bigg\|\sup_{k\in \mathbb Z} \frac 1{m_\lambda(I(\cdot, 2^{-k})){ 2^{-k\beta}}}\int_{I(\cdot, 2^{-k})} |f(y)-f_{I(\cdot, 2^{-k})}|dm_\lambda(y)\bigg\|_\lpz \\
&\quad \le 2 \bigg\|\sup_{k\in \mathbb Z} \frac 1{m_\lambda(I(\cdot, 2^{-k})){ 2^{-k\beta}}}\int_{I(\cdot, 2^{-k})} |f(\cdot)-f(y)|dm_\lambda(y)\bigg\|_\lpz\\
&\quad \lesssim \|f\|_{{\dot F}^{\beta,\infty}_{\lambda ,p}}.
\end{align*}

Conversely, we use the proof of \cite[Propositions 4.1]{WHYY} by replacing $f(\cdot)$ by $f_{I(\cdot, 2^{-k})}$ and get the following
estimate
$$\|f\|_{{\dot F}^{\beta,\infty}_{\lambda ,p}}\lesssim \bigg\|\sup_{k\in \mathbb Z} \frac 1{m_\lambda(I(\cdot, 2^{-k}))
{ 2^{-k\beta}}}\int_{I(\cdot, 2^{-k})} |f_{I(\cdot, 2^{-k})}-f(y)|dm_\lambda(y)\bigg\|_\lpz.$$
Hence, 
$$\|f\|_{{\dot F}^{\beta,\infty}_{\lambda ,p}}\sim \bigg\|\sup_{k\in \mathbb Z} \frac 1{m_\lambda(I(\cdot, 2^{-k}))
{ 2^{-k\beta}}}\int_{I(\cdot, 2^{-k})} |f_{I(\cdot, 2^{-k})}-f(y)|dm_\lambda(y)\bigg\|_\lpz
$$
and the proof is complete.
 \end{proof}

We now prove Theorem \ref{thm0}.

\begin{proof}[Proof of Theorem \ref{thm0}]
(a) $\Rightarrow$ (b): 
Fix $x\in \mathbb R^N, k\in \mathbb Z$ and $r=2^{-k}$. Set $I=I(x, r)$ and $f\in L^p(\mathbb{R}_+)$. 
Let $f_1=f\chi_{2I}$ and $f_2=f-f_1$. Observe that $[b,\riz]f=[b-b_I,\riz]f$, and we have:
\begin{align*}
&\frac1{m_\lambda(I) r^\beta}\int_I |[b,\riz]f(y)-([b,\riz]f)_I|dm_\lambda(y)  \\
&=\frac1{m_\lambda(I) r^\beta}\int_I |[b-b_I,\riz]f(y)-([b-b_I,\riz]f)_I|dm_\lambda(y)  \\
&\le \frac2{m_\lambda(I) r^\beta}\int_I |[b-b_I,\riz]f(y)-\riz((b-b_I)f_2)(x)|dm_\lambda(y)  \\
&\le \frac2{m_\lambda(I) r^\beta}\int_I |(b(y)-b_I)\riz f(y)|dm_\lambda(y) +
       \frac2{m_\lambda(I) r^\beta}\int_I |\riz((b-b_I)f_1)(y)|dm_\lambda(y)  \\
&\qquad + \frac2{r^\beta} \sup_{y\in I} |\riz((b-b_I)f_2)(y)-\riz((b-b_I)f_2)(x)| \\
&=: I+ II +III.
\end{align*}
We estimate these terms individually. For $1<t<p$, H\"older's inequality gives
\begin{align*}
&\frac1{m_\lambda(I) r^\beta}\int_I |(b(y)-b_I)\riz f(y)|dm_\lambda(y) \\
&\le \frac1{r^\beta} \Big(\frac1{m_\lambda(I)} \int_I |b(y)-b_I|^{t'} dm_\lambda(y) \Big)^{\frac1{t'}} 
           \Big(\frac1{m_\lambda(I)}\int_I |\riz f(y)|^tdm_\lambda(y)\Big)^{\frac 1t}\\ 
&\lesssim \|b\|_{\Lambda^\beta} M_t(\riz f)(x),
\end{align*}
where the last inequality follows from Theorem \ref{lip} and $M$ is the Hardy-Littlewood maximal function in space of
homogeneous type $(\Bbb R_+, dm_\lambda)$ and $M_t(f)=(M(f^t))^{\frac1t}$. Thus,
$$I\lesssim \|b\|_{\Lambda^\beta} M_t(\riz f)(x).$$

To estimate $II$, we use the boundedness of $\riz$ and H\"older's inequality to obtain for $1<t<p$
\begin{align*}
&\frac1{m_\lambda(I) r^\beta}\int_I |\riz((b-b_I)f_1)(y)|dm_\lambda(y) \\
&\le \frac1{m_\lambda(I) r^\beta}\|\riz((b-b_I)f_1)\|_{\lttz}m_\lambda(I)^{1-\frac 1t} \\
&\le m_\lambda(I)^{-\frac 1t}r^{-\beta} \|(b-b_I)f_1\|_{\lttz}.
\end{align*}
By H\"older's inequality with $1<s$ and $st<p$,
\begin{align*}
\frac1{m_\lambda(I)}\|(b-b_I)f_1\|_{\lttz}^{t}
&=\frac1{m_\lambda(I)}\int_{2I} |b(y)-b_I|^t|f|^tdm_\lambda(y) \\
&\le \Big(\frac1{m_\lambda(I)}\int_{2I} |b(y)-b_I|^{ts'}dm_\lambda(y) \Big)^{\frac1{s'}}
         \Big(\frac1{m_\lambda(I)}\int_{2I} |f|^{ts} dm_\lambda(y)\Big)^{\frac1s}.
\end{align*}
Now use Theorem \ref{lip} to get

\begin{equation}\label{norm}
\begin{aligned}
&\frac1{m_\lambda(I)}\int_{2I} |b(y)-b_I|^{ts'}dm_\lambda(y) \\
&\lesssim\frac1{m_\lambda(I)}\int_{2I} |b(y)-b_{2I}|^{ts'}dm_\lambda(y) 
       + |b_{2I}-b_I|^{ts'}\\
&\lesssim \frac1{m_\lambda(2I)}\int_{2I} |b(y)-b_{2I}|^{ts'}dm_\lambda(y)
     + \frac1{m_\lambda(I)}\int_I |b_{2I}-b(y)|^{ts'}dm_\lambda(y) \\
&\lesssim r^{\beta s't}\|b\|^{s't}_{\Lambda^\beta}.
\end{aligned}
\end{equation}
Thus, $II\lesssim \|b\|_{\Lambda^\beta}M_{ts}(f)(x).$

To estimate III, we use Theorem \ref{lip} to obtain that for $I^*\subseteq I(x,r)$, we have
$$|b_{I^*}-b_I|\lesssim \|b\|_{\Lambda^\beta}r^\beta\frac{m_\lambda(I)}{m_\lambda(I^*)}. $$
By Proposition \ref{t:RieszCZ} (c),
\begin{align*}
&|\riz((b-b_I)f_2)(y)-\riz((b-b_I)f_2)(x)|\\
&=\bigg|\int_{\mathbb R_+} (\riz(y,z)-\riz(x,z))(b(z)-b_I)f_2(z)dm_\lambda(z) \bigg|\\
&\lesssim\int_{(2I)^c} \frac{|x-y|}{|x-z|}\frac1{m_\lz(I(x, |x-z|))}|b(z)-b_I| |f(z)|dm_\lambda(z) \\
&\lesssim\sum_{j=1}^\infty \int_{2^jr<|x-z|\le 2^{j+1}r}  2^{-j}\frac1{m_\lz(I(x, |x-z|))}|b(z)-b_I| |f(z)|dm_\lambda(z) \\
&\lesssim\sum_{j=1}^\infty \frac1{m_\lambda(I(x,2^jr))}
      \int_{|x-z|\le 2^{j+1}r}  2^{-j}    \big(|b(z)-b_{2^jI}|+|b_{2^jI}-b_I|\big)|f(z)|dm_\lambda(z) \\
&\lesssim\sum_{j=1}^\infty \frac{2^{-j}}{m_\lambda(I(x,2^jr))}\int_{|x-z|\le 2^{j+1}r} 
          |b(z)-b_{2^jI}||f(z)|dm_\lambda(z)\\
 &\qquad + \sum_{j=1}^\infty2^{-j}\|b\|_{\Lambda^\beta} (2^jr)^\beta M(f)(x)j \\
 &\lesssim\Big(\|b\|_{\Lambda^\beta}r^\beta M_{ts}(f)(x) + \|b\|_{\Lambda^\beta}r^\beta M(f)(x)\Big) \sum_{j=1}^\infty 2^{-j(1-\beta)}j,
\end{align*}
where the last inequality follows as \eqref{norm}.
Thus, $III\lesssim\|b\|_{\Lambda^\beta}M_{ts}(f)(x) + \|b\|_{\Lambda^\beta}M(f)(x)$.

Putting these estimates together, we obtain
\begin{align*}
&\frac1{m_\lambda(I) r^\beta}\int_I |[b,\riz]f(y)-([b,\riz]f)_I|dm_\lambda(y) \\
&\qquad \lesssim \|b\|_{\Lambda^\beta}\big(M_t(\riz f)(x) + M_{ts}(f)(x) + M(f)(x)\big).
\end{align*}
Taking the supremum over all $I(x, r)$, then $L^p$-norm of both sides, and using Theorem \ref{prop 3.1}
we conclude that
$$\|[b,\riz]f\|_{{\dot F}^{\beta,\infty}_{\lambda ,p}}\lesssim \|b\|_{\Lambda^\beta}\|f\|_{\lpz}.$$
Thus, (a) $\Rightarrow$ (b) is proved.

(b) $\Rightarrow$ (a): 

Assume the commutator $[b,\riz]$ is a bounded operator from $\lpz$ to ${\dot F}^{\beta,\infty}_{\lambda ,p}(\mathbb{R}_+)$.
 For $I=I(x_0,r)$ with $x_0\in \mathbb R_+$ and $r>0$, we consider
$$
\frac 1{r^{\beta}m_\lambda(I)}\int_I |f(x)-f_I| dm_\lambda(x). 
$$

\begin{defn}\label{mfb}
Let $f$ be finite almost everywhere on $\mathbb R_+$. For $I\subseteq \mathbb R_+$ 
with $m_\lambda (I)<\infty$, we define a median value $m_f(I)$ of $f$ over $I$ to be a real number
satisfying
$$m_\lambda(\{x\in I : f(x)>m_f(I)\})\le \frac12m_\lambda (I)\qquad\mbox{and}\qquad 
m_\lambda(\{x\in I : f(x)<m_f(I)\})\le \frac12m_\lambda (I).$$
\end{defn}
Note that
\begin{align*}
[b, \riz]f(x)&=b(x)\riz f(x)-\riz(bf)(x)\big)=\int_{\mathbb R_+} (b(x)-b(y))\riz (x,y)f(y)dm_\lambda(y).
\end{align*}
Observe that if $r > x_0$, then
$I(x_0,r)=(x_0-r, x_0+r)\cap {\mathbb R}_+ =I(\frac{x_0+r}2,\frac{x_0+r}2).$ 
Therefore, without loss of generality, we may assume that $r \le x_0$.
We next consider the following two cases.

Case (1): $x_0\le 2r$. In this case, $m_\lambda (I)\sim x_0^{2\lambda}r\sim x_0^{2\lambda+1}$.
Let $\tilde x_0=x_0+2(K_1-1)r$, where $K_1$ is as in Proposition \ref{t:RieszCZ}. 
Then $K_1x_0\le \tilde x_0\le (2K_1+1)x_0$.
By applying Proposition \ref{t:RieszCZ} (f),
$$\riz (x_0,\tilde x_0)\gs \frac{x_0}{\tilde x_0^{2\lambda+2}}\sim \frac1{m_\lambda (I)}.$$
By \cite [Proposition 5.1]{dgklwy}, for every $x\in I$ and $y\in \widetilde I=I(\tilde x_0, r)$,  we have
$$|\riz(x,y)|\gs  \frac1{m_\lambda (I)}.$$

Case (2): $x_0> 2r$. In this case, $m_\lambda (I)\sim x_0^{2\lambda}r$. Let $\tilde x_0=x_0+2(K_0-1)r$, where $K_0$ is as in Proposition \ref{t:RieszCZ}. 
Then $x_0\le  \tilde x_0\le K_0x_0$.
By Proposition \ref{t:RieszCZ} (d), we get
$$\riz (x_0,\tilde x_0)\gs \frac1{\tilde x_0^{2\lambda}(\tilde x_0-x_0)}\sim \frac1{m_\lambda (I)}.$$
Again by \cite [Proposition 5.1]{dgklwy}, for every $x\in I$ and $y\in \widetilde I=I(\tilde x_0, r)$,  we have
$$|\riz(x,y)|\gs   \frac1{m_\lambda (I)}.$$

Set
 $$E_1:=\{y\in \widetilde I : b(y)\le m_b(\widetilde I)\}\qquad\mbox{and}\qquad 
    E_2:=\{y\in \widetilde I : b(y)\ge m_b(\widetilde I)\}$$
and further define
$$I_1:=\{y\in I : b(y)\ge m_b(\widetilde I)\}\qquad\mbox{and}\qquad 
    I_2:=\{y\in I : b(y)< m_b(\widetilde I)\}.$$
Definition \ref{mfb} shows that $m_\lambda(E_i)\ge \frac12m_\lambda(\widetilde I), i=1,2.$ 
Moreover, for $(x,y)\in I_i\times E_i, i=1,2,$
\begin{align*}
|b(x)-b(y)|&=|b(x)-m_b(\tilde I)+m_b(\widetilde I)-b(y)|\\
   &=|b(x)-m_b(\widetilde I)|+|m_b(\widetilde I)-b(y)|\ge |b(x)-m_b(\widetilde I)|.
\end{align*}
Hence, we have the following facts:
\begin{equation}\label{eq 3.2}
\begin{aligned}
& \mbox{(i) } I=I_1\cup I_2, \widetilde I= E_1\cup E_2 
    \mbox{ and }m_\lambda(E_i)\ge \frac12m_\lambda(\widetilde I), i=1,2;\\
& \mbox{(ii) } b(x)-b(y) \mbox{ does not change sign for all }(x,y)\in I_i\times E_i, i=1,2; \\
&\mbox{(iii) } |b(x)-m_b(\tilde I)|\le |b(x)-b(y)| \mbox{ for all }(x,y)\in I_i\times E_i, i=1,2.
\end{aligned}
\end{equation}
We also have that, for $(x,y)\in I_i\times E_i, i=1,2$, 
$$|\riz (x,y)|\gs    \frac1{m_\lambda (I)} \sim \frac 1{m_\lambda(I(\tilde x_0,r))}.$$

Choose a $C^\infty$ function $a$ with the following properties:
\begin{itemize}
\item [(1)] supp $a \subseteq I(x_0, 3r)$, $\int_{\mathbb R_+} a(x)\,dm_\lambda(x)=0$, $a(x)=1$ on $I$, and $|a(x)|\lesssim1 $;
\item[(2)]  $\|a\|_{\dot F^{-\beta, 1}_{\tbz ,p'}(\mathbb{R}_+)}\lesssim r^\beta m_\lambda(I)^{1/p'}$.
\end{itemize}

We provide a proof for the existence of such a function $a$. That is, if $a$ satisfies the conditions in (1) above, then we have the norm estimate in (2).

Recall that $(\mathbb R_+, |\cdot|, dm_\lambda)$ is the space of homogeneous type. 
 Let $\{\overline S_k\}_{k\in \Bbb Z}$ be Coifman's approximation to the identity.
 Then the kernels $\overline S_k(x,y)$ of $\overline S_k$ satisfy the following properties:
\begin{enumerate}
    \item[(i)] $\overline S_k(x,y)=\overline S_k(y,x);$
    \item[(ii)] $\overline S_k(x,y)=0$ if $|x-y|> c2^{-k}$\ \ and\ \ $\displaystyle |\overline S_k(x,y)|\lesssim \frac{1}{V_k(x)+V_k(y)};$
    \item[(iii)] $\displaystyle |\overline S_k(x,y)-\overline S_k(x',y)|\lesssim \frac{ 2^k|x-x'|}{V_k(x)+V_k(y)}$\qquad for $|x-x'|\leqslant  c_12^{-k};$\\[3pt]
    \item[(iv)] $\displaystyle |\overline S_k(x,y)-\overline S_k(x,y')|\lesssim \frac{ 2^k|y-y'|}{V_k(x)+V_k(y)}$\qquad for $|y-y'|\leqslant  c_12^{-k};$
    \item[(v)] $\displaystyle \big|[\overline S_k(x,y)-\overline S_k(x',y)]-[\overline S_k(x,y')-\overline S_k(x',y')]\big|\lesssim \frac{2^k|x-x'|2^k|y-y'|}{V_k(x)+V_k(y)}$
    \item[] for $|x-x'|\leqslant  c_12^{-k}$ and $|y-y'|\leqslant  c_12^{-k};$
    \item[(vi)] $\displaystyle \int_{\mathbb R_+} \overline S_k(x,y)dm_\lambda(x)=1\qquad \text{for all}\ y\in \mathbb R_+;$
    \item[(vii)] $\displaystyle \int_{\mathbb R_+} \overline S_k(x,y)dm_\lambda=1\qquad \text{for all}\ x\in \mathbb R_+.$
\end{enumerate}
 Here  $V_k(x)$ denotes the measure
$m_\lambda(I(x, 2^{-k}))$ for $k\in \Bbb Z$ and $x\in \mathbb R_+$.   And $V(x,y):=m_\lambda(I(x,|x-y|))$ for $x,y\in \mathbb R_+$.

Set ${\overline D}_k := {\overline S}_k - {\overline S}_{k-1}$.    
Note that ${\dot F}^{-\beta,1}_{\lambda ,p'}(\rrp,\, dm_\lz)$ is equivalent to the Triebel--Lizorkin space
${\dot F}^{-\beta,1}_{p',CW}(\rrp,\, dm_\lz)$ on spaces of homogeneous type $(\mathbb R_+, |\cdot|, dm_\lambda)$.  Thus, we have
$$\|a\|_{{\dot F}^{-\beta,1}_{\lambda ,p'}}\sim \|a\|_{{\dot F}^{-\beta,1}_{p',CW}}
:= \Big\|\Big\{\sum_{k\in \mathbb Z}\delta^{k\beta}|\overline D_k(a)|\Big\}\Big\|_{{L^{p'}(\rrp,\, dm_\lz)}}.$$

{
Choosing $k_0\in \mathbb Z$ with $2^{-k_0}\sim r$, we obtain
\begin{align*}
&\Big\|\Big\{\sum_{k\in \mathbb Z}2^{-k\beta}|\overline D_k(a)|\Big\}\Big\|_{{L^{p'}(\rrp,\, dm_\lz)}}\\
&\le \Big\|\Big\{\sum_{k=k_0+1}^\infty2^{-k\beta}|\overline D_k(a)|\Big\}\Big\|_{{L^{p'}(\rrp,\, dm_\lz)}}
       + \Big\|\Big\{\sum_{k=-\infty}^{k_0}2^{-k\beta}|\overline D_k(a)|\Big\}\Big\|_{{L^{p'}(\rrp,\, dm_\lz)}}\\
 &=:I_1+I_2.
\end{align*}
Assume that $k>k_0$, $k\in \mathbb N$. Since $\overline S_k(x,y)=0$ for $|x-y|> c2^{-k}$, we have that 
$\overline D_k(x,y)$ is supported in $|x-y|\leq c2^{-k+1}$.
Note also that  supp $a\subseteq I(x_0,3r)$, we have that 
$$\overline D_k(a)(x) = \int_{I(x_0,3r)}\overline D_k(x,y)a(y)dm_\lz(y)$$
is supported in $I(x_0,Cr)$ for some absolute positive constant $C$. 
Moreover, we also have
$$|\overline D_k(a)(x)| \leq \int_{I(x_0,3r)}|\overline D_k(x,y)||a(y)|dm_\lz(y) \lesssim \int_{I(x_0,3r)}|\overline D_k(x,y)|dm_\lz(y)\lesssim1.$$
Thus,
\begin{align*}
I_1     &\le \Bigg(\int_{I(x_0,Cr)}\bigg\{\sum_{k=k_0+1}^\infty2^{-k\beta}\big|\overline D_k(a)(x)\big|\bigg\}^{p'}dm_\lambda(x)\Bigg)^{1/p'}\\
&\lesssim \Bigg(\int_{I(x_0,Cr)}\bigg\{\sum_{k=k_0+1}^\infty2^{-k\beta}\bigg\}^{p'}dm_\lambda(x)\Bigg)^{1/p'}\\
&\lesssim r^\beta m_\lambda(I)^{1/p'}.
\end{align*}

To estimate $I_2$, we use the vanishing moment condition of $a(x)$ and smoothness $\overline D_k$ to get
\begin{align*}
I_2&\le  \sum_{k=-\infty}^{k_0}\Big\|2^{-k\beta}|\overline D_k(a)|\Big\|_{{L^{p'}(\rrp,\, dm_\lz)}}\\
    &\le \sum_{k=-\infty}^{k_0}\bigg(\int_{I(x_0,C\delta^k)}\bigg|2^{-k\beta}\int_{I(x_0,3r)} \Big(\overline D_k(x,y)-\overline D_k(x,x_0)\Big)a(y)dm_\lambda(y)\bigg|^{p'} dm_\lambda(x)\Bigg)^{1\over p'}\\
     &\lesssim \sum_{k=-\infty}^{k_0}\bigg(\int_{I(x_0,C\delta^k)}\bigg|2^{-k\beta}\int_{I(x_0,3r)} \frac{|y-x_0|2^k}{m_\lz(I(y,2^{-k})}dm_\lambda(y)\bigg|^{p'} dm_\lambda(x)\Bigg)^{1\over p'}\\
      &\lesssim \sum_{k=-\infty}^{k_0}\bigg(\int_{I(x_0,C2^{-k})}\bigg|2^{-k\beta}\int_{I(x_0,3r)} \frac{r2^k}{m_\lz(I(x_0,2^{-k})}dm_\lambda(y)\bigg|^{p'} dm_\lambda(x)\Bigg)^{1\over p'}\\
      &\lesssim \sum_{k=-\infty}^{k_0} r2^{-k(\beta-1)}\Big(\frac{m_\lz(I(x_0,r))}{m_\lz(I(x_0,2^{-k})}\Big
      )^{1-1/p'}m_\lz(I(x_0,r))^{1/p'}\\
      &\lesssim \sum_{k=-\infty}^{k_0} r2^{-k(\beta-1)}m_\lz(I(x_0,r))^{1/p'}\\
      &\lesssim r^\beta m_\lambda(I)^{1/p'}.
\end{align*}

}

We now continue with our proof.
Let $f_i=\chi_{E_i}, i=1,2.$ 
Then \eqref{eq 3.2} gives
\begin{align*}
&\frac1{r^{\beta}m_\lambda(I)}\sum_{i=1}^2\bigg|\int_I [b, \riz](f_i)(x) a(x) dm_\lambda(x)\bigg|\\
 &\ge \frac1{r^{\beta}m_\lambda(I)}\sum_{i=1}^2\bigg|\int_{I_i} [b, \riz](f_i)(x) a(x) dm_\lambda(x)\bigg| \\
 &= \frac1{r^{\beta}m_\lambda(I)}\sum_{i=1}^2\int_{I_i}\int_{E_i} |b(x)-b(y)||\riz(x,y)| |a(x)| dm_\lambda(y)dm_\lambda(x)\\
 &\gtrsim  \frac1{r^{\beta}m_\lambda(I)}\sum_{i=1}^2\int_{I_i}|b(x)-m_b(\widetilde I)|\frac 1{m_\lambda(I(\tilde x_
     0,r))}\int_{E_i} dm_\lambda(y)dm_\lambda(x)\\
  &\gtrsim  \frac1{r^{\beta}m_\lambda(I)}\sum_{i=1}^2\int_{I_i}|b(x)-m_b(\widetilde I)|dm_\lambda(x)\\
 &\gtrsim \frac 1{r^{\beta}m_\lambda(I)}\int_I |b(x)-b_I| dm_\lambda(x).
\end{align*}

On the other hand, from duality and the boundedness of $[b, \riz]$,
we deduce that
\begin{align*}
&\frac1{r^{\beta}m_\lambda(I)}\sum_{i=1}^2\bigg|\int_I [b, \riz](f_i)(x)a(x)dm_\lambda(x)\bigg|\\
&\lesssim \frac1{r^{\beta}m_\lambda(I)}\sum_{i=1}^2  \|[b, \riz]f_i\|_{ \dot{F}_{\lambda ,p}^{\beta,\infty}}
 \|a\|_{ \dot{F}_{\lambda ,p'}^{-\beta,1}} \\
&\lesssim \frac1{r^{\beta}m_\lambda(I)}\sum_{i=1}^2  \|[b, \riz]\|_{\lpz\to \dot{F}_{\lambda ,p}^{\beta,\infty}}\|f_i\|_{\lpz} \|a\|_{ \dot{F}_{\lambda ,p'}^{-\beta,1}} \\
&\lesssim \frac1{r^{\beta}m_\lambda(I)} \|[b, \riz]\|_{\lpz\to \dot{F}_{\lambda ,p}^{\beta,\infty}}m_\lambda(\widetilde I)^{1/p}r^\beta m_\lambda(I)^{1/p'}\\
&\lesssim \|[b, \riz]\|_{\lpz\to \dot{F}_{\lambda ,p}^{\beta,\infty}}.
\end{align*}
Therefore, we have
$$\frac 1{r^{\beta}m_\lambda(I)}\int_I |b(x)-b_I| dm_\lambda(x)\lesssim \|[b, \riz]\|_{\lpz\to \dot{F}_{\lambda ,p}^{\beta,\infty}}.$$ 
Hence, (b) $\Rightarrow$ (a) is done.
The proof is complete.
\end{proof}

\section{Endpoint Results}

\begin{proof}[Proof of Theorem \ref{thm bmo4.1}]
Suppose that $f\in L^\infty_c(\mathbb R_+, dm_\lambda)$. Since $f\in L^2(\mathbb R_+, dm_\lambda)$, we have
$[b, \riz]\in L^2(\mathbb R_+, dm_\lambda)$. Let $I$ be an interval in $(0,\infty)$. Decompose $f=f_1+f_2$, 
where $f_1=f\chi_{4I}$. As in \cite[page 680]{HST} we write
\begin{align*}
&[b, \riz]f(x)-([b, \riz]f)_I\\
&=\sigma_1(x)-(\sigma_1)_I+\sigma_2(x,u)+\sigma_4(x,u)-(\sigma_2(\cdot,u))_I+(\sigma_3(x,\cdot))_I,\ x,u\in I,
\end{align*}
where
\begin{align*}
\sigma_1(x)&=[b, \riz]f_1(x),\ x\in I;\\
\sigma_2(x,u)&=(b(x)-b_I)(\riz(f_2)(x)-\riz(f_2)(u)),\ x,u\in I;\\
\sigma_3(x,u)&=\riz((b-b_I)f_2)(u)-\riz((b-b_I)f_2)(x),\ x,u\in I;
\end{align*}
and $$\sigma_4(x,u)=(b(x)-b_I)[b, \riz]f_2(u),\ x\in I.$$
Since $[b, \riz]$ is bounded from $\ltz$ into itself, we obtain
\begin{align*}
\frac1{m_\lambda(I)}\int_I |\sigma_1(x)|dm_\lambda(x)
&\le \bigg(\frac1{m_\lambda(I)}\int_I |\sigma_1(x)|^2dm_\lambda(x)\bigg)^{1/2} \\
&\lesssim\frac1{(m_\lambda(I))^{1/2}}\bigg(\int_{4I} |f(x)|^2dm_\lambda(x)\bigg)^{1/2}\|b\|_{{\rm BMO}(\mathbb R_+, dm_\lambda)}\\
&\lesssim\|f\|_{L^\infty(\mathbb R_+, dm_\lambda)}\|b\|_{{\rm BMO}(\mathbb R_+, dm_\lambda)}\bigg(\frac{m_\lambda(4I)}{m_\lambda(I)}\bigg)^{1/2}\\
&\lesssim\|f\|_{L^\infty(\mathbb R_+, dm_\lambda)}\|b\|_{{\rm BMO}(\mathbb R_+, dm_\lambda)}.
\end{align*}
According to Proposition \ref{t:RieszCZ} (c) and using the doubling property of $m_\lambda$, we get
\begin{align*}
|\riz(f_2)(x)-\riz(f_2)(u)|
&\lesssim\sum_{j=2}^\infty\int_{(2^{j+1}I)\setminus (2^jI)} \frac{|x-u|}{|x-y|}\frac{|f(y)|}{m_\lambda(B(x,|x-y|))}dm_\lambda(y)\\
&\lesssim\sum_{j=2}^\infty \frac{m_\lambda(2^{j+1}I)}{2^jm_\lambda(B(x,2^jr_0))}\|f\|_{L^\infty(\mathbb R_+, dm_\lambda)}\\
&\lesssim\|f\|_{L^\infty(\mathbb R_+, dm_\lambda)}\sum_{j=2}^\infty2^{-j}\\
&\lesssim\|f\|_{L^\infty(\mathbb R_+, dm_\lambda)},\ x,u\in I.
\end{align*}
Then, 
\begin{align*}
\frac1{m_\lambda(I)}\int_I |\sigma_2(x,u)|dm_\lambda(x)
&\le  C\|f\|_{L^\infty(\mathbb R_+, dm_\lambda)}\frac1{m_\lambda(I)}\int_I |b(x)-b_I|dm_\lambda(x)\\
&\lesssim\|f\|_{L^\infty(\mathbb R_+, dm_\lambda)}\|b\|_{{\rm BMO}(\mathbb R_+, dm_\lambda)},\ u\in I.
\end{align*}
Using again Proposition \ref{t:RieszCZ} (c) we obtain that for $x,u\in I$,
\begin{align*}
|\sigma_3(x,u)|
&\le \int_{\mathbb R_+ \setminus (4I)} |b(y)-b_I||\riz(x,y)-\riz(u,y)|dm_\lambda(y)\|f\|_{L^\infty(\mathbb R_+, dm_\lambda)}\\
&\lesssim\sum_{j=2}^\infty\int_{(2^{j+1}I)\setminus (2^jI)}  |b(y)-b_I| \frac{|x-u|}{|x-y|}\frac{dm_\lambda(y)}{m_\lambda(B(x,|x-y|))}
                       \|f\|_{L^\infty(\mathbb R_+, dm_\lambda)}\\
&\lesssim \|f\|_{L^\infty(\mathbb R_+, dm_\lambda)}\sum_{j=2}^\infty 2^{-j}\frac1{m_\lambda(B(x,2^jr_0))}\int_{(2^{j+1}I)}|b(y)-b_I|dm_\lambda(y) \\
&\lesssim \|f\|_{L^\infty(\mathbb R_+, dm_\lambda)}
           \sum_{j=2}^\infty 2^{-j}\frac1{m_\lambda(2^jI)}\int_{(2^{j+1}I)}|b(y)-b_I|dm_\lambda(y) \\
&\lesssim\|f\|_{L^\infty(\mathbb R_+, dm_\lambda)}\|b\|_{{\rm BMO}(\mathbb R_+, dm_\lambda)} \sum_{j=2}^\infty j2^{-j}\\
&\lesssim\|f\|_{L^\infty(\mathbb R_+, dm_\lambda)}\|b\|_{{\rm BMO}(\mathbb R_+, dm_\lambda)}.
\end{align*}
Then,
$$
\frac1{m_\lambda(I)}\int_I |\sigma_3(x,u)|dm_\lambda(x)
\lesssim\|f\|_{L^\infty(\mathbb R_+, dm_\lambda)}\|b\|_{{\rm BMO}(\mathbb R_+, dm_\lambda)},\ \  u\in I.
$$
We choose, for every interval $I \subseteq \mathbb R_+, u_I\in I$, we have that
\begin{align*}
&\bigg|\sup_{I \subseteq \mathbb R_+} \frac1{m_\lambda(I)}\int_I |[b, \riz](f)(x)-([b, \riz](f))_I|dm_\lambda(x) 
   - \sup_{I \subseteq \mathbb R_+} \frac1{m_\lambda(I)}\int_I |\sigma_4(x,u_I)|dm_\lambda(x)\bigg|\\
&\le \sup_{I \subseteq \mathbb R_+} \frac1{m_\lambda(I)}\int_I |\sigma_1(x)-(\sigma_1)_I+\sigma_2(x,u_I)-(\sigma_2(\cdot,u_I))_I+(\sigma_3(x,\cdot))_I|dm_\lambda(x) \\
&\lesssim\|f\|_{L^\infty(\mathbb R_+, dm_\lambda)}\|b\|_{{\rm BMO}(\mathbb R_+, dm_\lambda)}.
\end{align*}
We obtain that 
$\|[b, \riz](f)\|_{{\rm BMO}(\mathbb R_+, dm_\lambda)}\lesssim\|f\|_{L^\infty(\mathbb R_+, dm_\lambda)}\|b\|_{{\rm BMO}(\mathbb R_+, dm_\lambda)}$
if and only if,
\begin{equation}\label{H2}
\sup_{I \subseteq \mathbb R_+} \frac1{m_\lambda(I)}\int_I |\sigma_4(x,u_I)|dm_\lambda(x)
\lesssim\|f\|_{L^\infty(\mathbb R_+, dm_\lambda)}\|b\|_{{\rm BMO}(\mathbb R_+, dm_\lambda)}
\end{equation}
for some $u_I\in I$.

Estimate \eqref{H2} can be written as
\begin{align*}
&\sup_{I \subseteq \mathbb R_+} \frac1{m_\lambda(I)}\int_I |b(y)-b_I|dm_\lambda(y)
\bigg|\int_{\mathbb R_+ \setminus (4I)} \riz(u_I, z)f(z)dm_\lambda(z) \bigg|\\
&\lesssim\|f\|_{L^\infty(\mathbb R_+, dm_\lambda)}\|b\|_{{\rm BMO}(\mathbb R_+, dm_\lambda)}
\end{align*}
for some $u_I\in I$.

Let $K_0$ be the constant appearing in Proposition \ref{t:RieszCZ}. 
Write 
\begin{align*}
&\int_{\mathbb R_+ \setminus (4I)} \riz(u_I,z)f(z)dm_\lambda(z)\\
&=\bigg(\int_{\mathbb R_+ \setminus (4I) \atop {K_0^{-1}u_I\le z\le K_0 u_I}}
      +\int_{\mathbb R_+ \setminus (4I) \atop {0< z\le K_0^{-1} u_I}}
      +\int_{\mathbb R_+ \setminus (4I) \atop {K_0 u_I< z}} \bigg)   \riz(u_I,z)f(z)dm_\lambda(z)\\
&=: J_1(I)+J_2(I)+J_3(I),\qquad I\subseteq \mathbb R_+\mbox{ interval and }u_I\in I.
\end{align*}
According to Proposition \ref{t:RieszCZ} (a), we get, for every interval $I\subseteq \mathbb R_+$
and $u_I\in I$,
$$|J_2(I)|\lesssim\int_0^{K_0^{-1}u_I} \frac{dm_\lambda(z)}{u_I^{2\lambda+1}}
      \|f\|_{L^\infty(\mathbb R_+, dm_\lambda)}
      \lesssim\|f\|_{L^\infty(\mathbb R_+, dm_\lambda)}$$
and
$$|J_3(I)|\lesssim\int_{K_0 u_I}^\infty \frac{dm_\lambda(z)}{u_I^{2\lambda+2}}
      \|f\|_{L^\infty(\mathbb R_+, dm_\lambda)}
      \lesssim\|f\|_{L^\infty(\mathbb R_+, dm_\lambda)}.$$
It follows that
\begin{align*}
&\bigg|\sup_{I\subseteq \mathbb R_+} \frac1{m_\lambda(I)}\int_I |b(y)-b_I|dm_\lambda(y)
\bigg|\int_{\mathbb R_+ \setminus (4I)} \riz(u_I,z)f(z)dm_\lambda(z)\bigg|\\
&\qquad-\sup_{I\subseteq \mathbb R_+} \frac1{m_\lambda(I)}\int_I |b(y)-b_I|dm_\lambda(y)
\bigg|\int_{\mathbb R_+ \setminus (4I) \atop {K_0^{-1}u_I\le z\le K_0 u_I}} \riz(u_I,z)f(z)dm_\lambda(z)\bigg|\bigg|\\
&\lesssim\|f\|_{L^\infty(\mathbb R_+, dm_\lambda)}\|b\|_{{\rm BMO}(\mathbb R_+, dm_\lambda)},
\end{align*}
where $u_I\in I$, for every interval $I\subseteq \mathbb R_+$.

We deduce that $\|[b, \riz](f)\|_{{\rm BMO}(\mathbb R_+, dm_\lambda)}\lesssim\|f\|_{L^\infty(\mathbb R_+, dm_\lambda)}\|b\|_{{\rm BMO}(\mathbb R_+, dm_\lambda)}$
if and only if 
\begin{equation}\label{eq H3}
\begin{aligned}
&\sup_{I\subseteq \mathbb R_+} \frac1{m_\lambda(I)}\int_I |b(y)-b_I|dm_\lambda(y)
\bigg|\int_{\mathbb R_+ \setminus (4I) \atop {K_0^{-1}u_I\le z\le K_0 u_I}} \riz(u_I,z)f(z)dm_\lambda(z)\bigg|\\
&\lesssim\|f\|_{L^\infty(\mathbb R_+, dm_\lambda)}\|b\|_{{\rm BMO}(\mathbb R_+, dm_\lambda)},
\end{aligned}
\end{equation}
where $u_I\in I$, for every interval $I\subseteq \mathbb R_+$.

Note that if \eqref{eq H3} holds for some family $\{u_I : I\subseteq \mathbb R_+\}$, then 
\eqref{eq H3} holds for every family $\{u_I : I\subseteq \mathbb R_+\}$.

Suppose that $I=(0,2x_0)$, with $x_0>0$. Then, $4I=(0,5x_0)$. We have that
$(\mathbb R_+ \setminus (4I))\cap \{z\in \mathbb R_+: K_0^{-1}u_I\le z\le K_0 u_I\}=\varnothing$
for every $u_I\in I$. 

We {now} consider $I=(x_0-r_0, x_0+r_0)$ with $0<r_0<x_0\le 4x_0$.
Then, $4I=(0, x_0+4r_0)$. We take $u_I=x_0$ and observe that
$$\varnothing\ne (\mathbb R_+ \setminus (4I))\cap \{z\in \mathbb R_+: K_0^{-1}u_I\le z\le K_0 u_I\}
=[x_0+4r_0, K_0x_0]$$
if and only if $x_0+4r_0\le K_0x_0.$ 

According to \cite[(1.3)]{yy}, we get
\begin{align*}
|J_1(I)||
&\lesssim\|f\|_{L^\infty(\mathbb R_+, dm_\lambda)}\int_{x_0+4r_0}^{K_0x_0} 
      \frac{dm_\lambda(z)}{m_\lambda(B(z,|x_0-z|))} \\
&\lesssim\|f\|_{L^\infty(\mathbb R_+, dm_\lambda)}\int_{x_0+4r_0}^{K_0x_0} \frac{dz}{z-x_0}\\
&\lesssim\|f\|_{L^\infty(\mathbb R_+, dm_\lambda)}\log\Big(\frac{(K_0-1)x_0}{4r_0}\Big)\\
&\lesssim\|f\|_{L^\infty(\mathbb R_+, dm_\lambda)}.
\end{align*}
Since $b\in {\rm BMO}(\mathbb R_+)$, the inequality \eqref{eq H3} holds if and only if
\begin{align*}
&\sup_{I=(x_0-r_0,x_0+r_0) \atop 0<4r_0<x_0} \frac1{m_\lambda(I)}\int_I |b(y)-b_I|dm_\lambda(y)
\bigg|\int_{\mathbb R_+ \setminus (4I) \atop {K_0^{-1}u_I\le z\le K_0 u_I}} \riz(u_I,z)f(z)dm_\lambda(z)\bigg|\\
&\qquad\lesssim\|f\|_{L^\infty(\mathbb R_+, dm_\lambda)}\|b\|_{{\rm BMO}(\mathbb R_+, dm_\lambda)}.
\end{align*}
Suppose that $I=(x_0-r_0, x_0+r_0)$ with $0<4r_0<x_0$. If $K_0x_0>x_0+4r_0$ we get
\begin{align*}
\bigg|\int_{x_0+4r_0}^{K_0x_0} \riz(x_0,z)f(z)dm_\lambda(z) \bigg|
&\lesssim\|f\|_{L^\infty(\mathbb R_+, dm_\lambda)}\log\Big(\frac{(K_0-1)x_0}{4r_0}\Big)\\
&\lesssim\|f\|_{L^\infty(\mathbb R_+, dm_\lambda)}\log\Big(\frac{x_0}{r_0}\Big).
\end{align*} 
If $K_0^{-1}x_0<x_0-4r_0$, then
$$\bigg|\int_{K_0^{-1}x_0}^{x_0-4r_0} \riz(x_0,z)f(z)dm_\lambda(z) \bigg|
\lesssim\|f\|_{L^\infty(\mathbb R_+, dm_\lambda)}\log\Big(\frac{x_0}{r_0}\Big).$$
We deduce that $[b, \riz]$ is bounded from $L^\infty(\mathbb R_+, dm_\lambda)$ into
${\rm BMO}(\mathbb R_+, dm_\lambda)$ provided that $b\in {\rm BMO}(\mathbb R_+, dm_\lambda)$
and
\begin{equation}\label{eq H4}
\sup_{I=(x_0-r_0,x_0+r_0) \atop 0<4r_0<x_0} \log\Big(\frac{x_0}{r_0}\Big)\frac1{m_\lambda(I)}\int_I |b(y)-b_I|dm_\lambda(y)<\infty.
\end{equation}
Suppose now that $I=(x_0-r_0, x_0+r_0)$ with $0<4r_0<x_0$ such that
$$ (\mathbb R_+ \setminus (4I))\cap \{z\in \mathbb R_+: K_0^{-1}u_I\le z\le K_0 u_I\}\ne \varnothing.$$
We have that $r_0<(K_0-1)x_0/4$. According to Proposition \ref{t:RieszCZ} (d), we deduce that
\begin{align*}
&\bigg|\int_{\mathbb R_+ \setminus (4I)} \riz(x_0,z)\chi_{(x_0+4r_0, K_0x_0)}(z)dm_\lambda(z)\bigg|\\
&\quad\ge C\int_{x_0+4r_0}^{K_0x_0} \frac1{(x_0z)^\lambda(z-x_0)} z^{2\lambda}dz\\
&\quad\ge C\log\Big(\frac{(K_0-1)x_0}{4r_0}\Big)\\
&\quad\ge C \log\Big(\frac{x_0}{r_0}\Big).
\end{align*}
It follows that if $b\in {\rm BMO}(\mathbb R_+, dm_\lambda)$ and if  $[b, \riz]$ is bounded from $L^\infty(\mathbb R_+, dm_\lambda)$ into
${\rm BMO}(\mathbb R_+, dm_\lambda)$, then \eqref{eq H4} holds.
\end{proof}

\begin{proof}[Proof of Theorem \ref{thm bmo4.2}]
Suppose that $a$ is a $(1, 2)_{\tbz }$-atom with supp $a\subseteq I=(x_0-r_0,x_0+r_0)$
where $0<r_0\le x_0$. Since $b\in {\rm BMO}(\mathbb R_+, m_\lambda), 
b\in L^1(\mathbb R_+, m_\lambda)$ and $[b, \riz]$ is well defined. We can write (see \cite[pages 688]{HST})
\begin{align*}
[b, \riz](a)(x)
&=\chi_{4I}(x)[b, \riz](a)(x) + \chi_{\mathbb R_+ \setminus (4I)}(b(x)-b_I)\riz(a)(x) \\
&\quad +\chi_{\mathbb R_+ \setminus (4I)}(x)\riz(a(b-b_I))(x) \\
&=\chi_{4I}(x)[b, \riz](a)(x) + \chi_{\mathbb R_+ \setminus (4I)}(b(x)-b_I)\riz(a)(x) \\
&\quad + \chi_{\mathbb R_+ \setminus (4I)}(x)\int_I (\riz(x,y)-\riz(x,u))(b(y)-b_I)a(y)dm_\lambda(y)\\
&\quad+ \chi_{\mathbb R_+ \setminus (4I)}(x)-\riz(x,u)\int_I (b(y)-b_I)a(y)dm_\lambda(y) \\
&=:T_1(a)(x)+T_2(a)(x)+T_3(a)(x,u)+T_4(a)(x,u),\ u\in I \mbox{ and }x\in \mathbb R_+. 
\end{align*}
Since $[b, \riz]$ is bounded from $\ltz$ into itself, it follows that
\begin{align*}
\|T_1(a)\|_{L^1(\mathbb R_+, m_\lambda)} 
&=\int_{4I} |[b, \riz](a)(x)|dm_\lambda(x) 
      \le (m_\lambda(4I))^{1/2}\|[b, \riz](a)\|_{L^2(\mathbb R_+, dm_\lambda)} \\
&\lesssim(m_\lambda(4I))^{1/2}\|a\|_{L^2(\mathbb R_+, dm_\lambda)} \|b\|_{{\rm BMO}(\mathbb R_+, dm_\lambda)}\\
&  \lesssim\|b\|_{{\rm BMO}(\mathbb R_+, dm_\lambda)}.
\end{align*}
For any $u\in I$ we have that
$$T_2(a)(x)=\chi_{\mathbb R_+ \setminus (4I)}(x)(b(x)-b(y))
\int_I (\riz(x,z)-\riz(x,u))a(y)dm_\lambda(y),\ x\in  \mathbb R_+.$$
According to Proposition \ref{t:RieszCZ} (c) and taking $u\in I$, we obtain
\begin{align*}
\|T_2(a)\|_{L^1(\mathbb R_+, dm_\lambda)} 
&\le \int_{\mathbb R_+ \setminus (4I)} |b(x)-b_I|\int_I |\riz(x,y)-\riz(x,u)||a(y)|dm_\lambda(y) dm_\lambda(x) \\
&\lesssim\int_{\mathbb R_+ \setminus (4I)} |b(x)-b_I|\int_I \frac{|y-u|}{|x-y|m_\lambda(I(x,|x-y|))}
      |a(y)|dm_\lambda(y) dm_\lambda(x)\\
&\lesssim r_0\sum_{j=2}^\infty \int_{(2^{j+1}I)\setminus(2^jI)} |b(x)-b_I|\int_I \frac{|a(y)|}{|x-y|m_\lambda(I(x,|x-y|))}
      dm_\lambda(y) dm_\lambda(x).
\end{align*}
For every $j\in \mathbb Z, j\ge 2$, $x\in (2^{j+1}I)\setminus(2^jI)$ and $y\in I$ we get
$$2^jI\subseteq I(y,2^{j+1}r_0)\subseteq I(y, 4|x-y|)\subseteq I(x, 5|x-y|),$$
and, since $m_\lambda$ is doubling,
$$m_\lambda(2^jI)\lesssim2^jm_\lambda(I(x, |x-y|)).$$
It follows that
\begin{align*}
\|T_2(a)\|_{L^1(\mathbb R_+, dm_\lambda)} 
&\lesssim r_0\sum_{j=2}^\infty \int_{(2^{j+1}I)} \frac{|b(x)-b_I|}{2^jr_0m_\lambda(2^jI)}
     dm_\lambda(x) \\
&\lesssim \sum_{j=2}^\infty \frac1{2^jr_0m_\lambda(2^jI)} \int_{(2^{j+1}I)}|b(x)-b_I|dm_\lambda(x)\\
&\lesssim \sum_{j=2}^\infty j2^{-j}\|b\|_{{\rm BMO}(\mathbb R_+, dm_\lambda)}\\
&\lesssim \|b\|_{{\rm BMO}(\mathbb R_+, dm_\lambda)}.
\end{align*}
By using again Proposition \ref{t:RieszCZ} (c) we can write
\begin{align*}
&\|T_3(a)(\cdot,u)\|_{L^1(\mathbb R_+, dm_\lambda)} \\
&\le \int_{\mathbb R_+ \setminus (4I)}\int_I |\riz(x,y)-\riz(x,u)||b(y)-b_I||a(y)|dm_\lambda(y) dm_\lambda(x)\\
&\le \sum_{j=2}^\infty \int_{(2^{j+1}I)\setminus(2^jI)}\int_I \frac{|y-u|}{|x-y|m_\lambda(I(x,|x-y|))}
  |b(y)-b_I||a(y)|dm_\lambda(y) dm_\lambda(x)\\    
&\lesssim r_0\sum_{j=2}^\infty \int_{(2^{j+1}I)}\int_I \frac1{2^jr_0m_\lambda(2^jI)}   |b(y)-b_I||a(y)|dm_\lambda(y) dm_\lambda(x)\\   
&\lesssim \sum_{j=2}^\infty \frac{m_\lambda(2^{j+1}I)}{2^jm_\lambda(2^jI)}\|b-b_I\|_{L^2(\mathbb R_+, dm_\lambda)}\|a\|_{\ltz}\\
&\lesssim \|b\|_{{\rm BMO}(\mathbb R_+, dm_\lambda)}, \ u\in I.
\end{align*}
Note that the above implied constants do not depend on $a$ and $u\in I$.
We have proved that, for every $u\in I$,
\begin{equation}\label{eq H6}
\|[b,\riz](a)-T_4(a)(\cdot,u)\|_{L^1(\mathbb R_+, dm_\lambda)} \lesssim \|b\|_{{\rm BMO}(\mathbb R_+, dm_\lambda)}.\end{equation}
According to \cite[Theorem 1.1]{YZ},  $[b, \riz]$ can be extended from span$\{(1, 2)_{\tbz }$-atom$\}$ to $H^1(\mathbb R_+, dm_\lambda)$ as a bounded operator from $H^1(\mathbb R_+, dm_\lambda)$ into 
$L^1(\mathbb R_+, dm_\lambda)$ if, and only if, there exists a positive constant such that, for every 
$(1, 2)_{\tbz }$-atom with supp $a\subseteq I$ such that
$$\|T_4(a)(\cdot,u)\|_{L^1(\mathbb R_+, dm_\lambda)} \lesssim 1,$$
equivalently, since $\int_I a(x) dm_\lambda(x)=0$,
\begin{equation}\label{eq H7}
\bigg| \int_{\mathbb R_+ \setminus (4I)} \riz(x,u) dm_\lambda(x)\bigg|
\bigg|\int_I b(y)a(y)dm_\lambda(y)\bigg|\lesssim 1.
\end{equation}
Note that by using \eqref{eq H6} we can see that in the above property we can change 
``there exists $u\in I$" by ``for every $u\in I$". In the sequel we fix $u=x_0$, where $I=(x_0-r_0,x_0+r_0)$ with $0<r_0\le x_0$.

According to Proposition \ref{t:RieszCZ} (a), and \cite[(1.4)]{yy} we obtain
$$|\riz(x,z)|\lesssim \frac 1{m_\lambda(B(x,|x-z|))}\lesssim \frac 1{m_\lambda(B(z,|x-z||))}
      \lesssim \frac 1{z^{2\lambda}|x-z|},\ x,z\in \mathbb R_+.$$
If $x_0>4r_0$, then
$$\bigg|\int_0^{x_0-4r_0} \riz(x,x_0)dm_\lambda(x)\bigg|
    \lesssim \int_0^{x_0-4r_0} \frac{x^{2\lambda}}{x_0^{2\lambda}|x-x_0|}dx\lesssim \frac{(x_0-4r_0)^{2\lambda+1}}{x_0^{2\lambda}r_0}.$$
If $K_1x_0>x_0+4r_0$, it follows that
$$\bigg|\int_{x_0+4r_0}^{K_1x_0} \riz(x,x_0)dm_\lambda(x)\bigg|
    \lesssim \int_{x_0+4r_0}^{K_1x_0} \frac{x^{2\lambda}}{x_0^{2\lambda}|x-x_0|}dx\lesssim \frac{x_0}{r_0}.$$
Using Proposition \ref{t:RieszCZ} (g), we get
$$\bigg|\int_{K_1x_0}^\infty \riz(x,x_0)dm_\lambda(x) \bigg|\gtrsim \int_{K_1x_0}^\infty  \frac1{x^{2\lambda+1}}x^{2\lambda}dx=\infty.$$
The above estimates allows us to deduce that
$$\bigg|\int_{\mathbb R_+ \setminus (4I)} \riz(x,x_0)dm_\lambda(x) \bigg|=\infty.$$
If \eqref{eq H7} holds for $I=(x_0-r_0,x_0+r_0)$ with $0<r_0\le x_0$, and $u=x_0$ and
a $(1, 2)_{\tbz }$-atom $a$ with supp $(a)\subseteq I$, then 
$$\int_I b(y)a(y)dm_\lambda(y)=0.$$

We have proved that $[b,\riz]$ can be extended from
form span$\{(1, 2)_{\tbz }$-atom$\}$ to $H^1(\mathbb R_+, dm_\lambda)$ as a bounded operator from $H^1(\mathbb R_+, dm_\lambda)$ into 
$L^1(\mathbb R_+,d m_\lambda)$ if, and only if, for every $(1, 2)_{\tbz }$-atom $a$,
$\int_0^\infty b(y)a(y)dm_\lambda(y)=0$
since $b\in {\rm BMO}(\mathbb R_+, dm_\lambda)=\big(H^1(\mathbb R_+, dm_\lambda)\big)'$, if $$\int_0^\infty b(y)a(y)dm_\lambda(y)=0$$ for every $(1, 2)_{\tbz }$-atom $a$, then $b=0$ as an element of 
${\rm BMO}(\mathbb R_+)$, that is $b$ is a constant for almost all $\mathbb R_+$.
\end{proof}


\section{The proof of Theorem \ref{thm 1.7}}


{
Let $K^\alpha(x,y)$ be the kernel of $\tbz^{-\alpha/2}$.
Then 
$$
K^\alpha(x,y)\sim (x+y)^{-2\lambda}
\begin{cases} |x-y|^{\alpha-1}, &  0<\alpha<1,\\
   \ln \frac{2(x+y)}{|x-y|}, &\alpha=1,\\  
   (x+y)^{\alpha-1}, &1<\alpha<{\bf Q}.
\end{cases}
$$
The above estimates can be found in \cite[Theorem 2.1]{NS}. It follows that

\begin{lem}
Let $0<\alpha<1$. We have 
\begin{equation}\label{frac size}
K^\alpha(x,y)\sim \frac{|x-y|^\alpha}{m_\lambda(I(x,|x-y|))}.
\end{equation}
\end{lem}

\begin{proof}
Suppose $x,y\in 
\mathbb R_+$ with $x<y$. If $x>2|x-y|$, then 
$$\frac1{(x+y)^{2\lambda}}\frac1{|x-y|^{1-\alpha}}\sim
 \frac1{x^{2\lambda}}\frac1{|x-y|^{1-\alpha}}\sim \frac{|x-y|^\alpha}{m_\lambda(I(x,|x-y|))}.$$
If $x\le 2|x-y|$, then $x+y\sim |x-y|$ which gives
$$\frac1{(x+y)^{2\lambda}}\frac1{|x-y|^{1-\alpha}}\sim
 \frac1{|x-y|^{2\lambda}}\frac1{|x-y|^{1-\alpha}}\sim \frac{|x-y|^\alpha}{m_\lambda(I(x,|x-y|))}.$$
Suppose $x,y\in 
\mathbb R_+$ with $y<x$, then we have 
$$\frac1{(x+y)^{2\lambda}}\frac1{|x-y|^{1-\alpha}}\sim \frac{|x-y|^\alpha}{m_\lambda(I(y,|x-y|))}\sim \frac{|x-y|^\alpha}{m_\lambda(I(x,|x-y|))}.$$
\end{proof}

Using the arguments in the proof of  \cite[Theorem 2.1]{NS}, we also have the smooth conditions as follows.

\begin{lem}
Let $0<\alpha<1$. If
$x,x',y\in\mathbb R_+$ with  $|x-x'|<|x-y|/2$, then
\begin{equation}\label{frac smooth}
|K^\alpha(x,y)-K^\alpha(x',y)|
\lesssim   \frac{|x-x'|}{|x-y|} \frac{|x-y|^\alpha}{m_\lambda(I(x,|x-y|))}.
\end{equation}
\end{lem}

\begin{proof}
We first demonstrate that
\begin{equation}\label{frac smooth 1}
\Big|\frac\partial{\partial x} K^\alpha(x,y)\Big|\lesssim \frac{|x-y|^\alpha}{|x-y|^2(x+y)^{2\lambda}}, \qquad x,y\in(0,\infty),x\ne y.
\end{equation}
Let $I_\lambda$ denote the Bessel function of the second kind of order $\lambda$. We have that
$$K^\alpha(x,y)=\frac1{\Gamma(\alpha)}\int_0^\infty W^\lambda_t(x,y)t^{\frac\alpha2-1}dt,  \qquad x,y\in(0,\infty),x\ne y,$$
where $W^\lambda_t(x,y)$ is the heat kernel associated with $\tbz$ with the explicit expression as follows (see for example \cite{bdt})
$$W^\lambda_t(x,y)=\frac1{2t}\exp\Big(-\frac{x^2+y^2}{4t}\Big)(xy)^{-\lambda+\frac12}I_{\lambda-\frac12}\Big(\frac{xy}{2t}\Big), \qquad t,x,y\in(0,\infty).$$

Since $\frac\partial{\partial_z}(z^{-\mu}I_\mu(z))=z^{-\mu}I_{\mu+1}(z), z\in (0,\infty)$, we obtain
$$\frac\partial{\partial x}  W^\lambda_t(x,y)= \frac1{2t}\exp\Big(-\frac{x^2+y^2}{4t}\Big)\Big[-\frac x{2t}(xy)^{-\lambda+\frac12}I_{\lambda-\frac12}\Big(\frac{xy}{2t}\Big)+\frac y{2t}(xy)^{-\lambda+\frac12}I_{\lambda+\frac12}\Big(\frac{xy}{2t}\Big)\Big]$$
for $t,x,y\in(0,\infty)$.
Using the asymptotic expansion of $I_\mu$ and the asymptotic $I_\mu(z)\sim \frac{z^\mu}{2^\mu\Gamma(\mu+1)},$ as $z\to 0^+,$
we deduce that
\begin{align*}
\Big|\frac\partial{\partial x}  W^\lambda_t(x,y)\Big|
\lesssim\begin{cases} \frac{\displaystyle x+y}{\displaystyle  t^{\lambda+\frac32}}\exp\Big(-\frac{\displaystyle x^2+y^2}{\displaystyle 4t}\Big), & xy\le t, \\[15pt]
\big(\frac{\displaystyle |x-y|}{\displaystyle \sqrt t}+\frac{\displaystyle x+y}{\displaystyle \sqrt t}\big)\ \frac{\displaystyle 1}{\displaystyle t}\ \exp\Big(-\frac{\displaystyle (x-y)^2}{\displaystyle 4t}\Big)(xy)^{-\lambda}, & xy\ge t,
\end{cases}\qquad t,x,y\in(0,\infty).
\end{align*}
Thus, 
\begin{align*}
\Big|\frac\partial{\partial x} K^\alpha(x,y)\Big|
&\lesssim (x+y)\bigg[(xy)^{-\lambda}\int_0^{xy} \exp\Big(-\frac{(x-y)^2}{4t}\Big) t^{\frac{\alpha-4}2}dt
\Big(\frac1{x+y}+\frac1{\sqrt{xy}}\Big)\\
&\hskip2cm+\int_{xy}^\infty \exp\Big(-\frac{x^2+y^2}{4t}\Big)t^{\frac\alpha2-\lambda-\frac52}dt\bigg]\\
&\lesssim (x+y)(xy)^{-\lambda+\frac{\alpha-3}2}\bigg[\int_0^1 \exp\Big(-\frac{(x-y)^2}{4xyu}\Big)u^{\frac{\alpha-4}2}du\Big(\frac1{x+y}+\frac1{\sqrt{xy}}\Big)\\
&\hskip2cm+\int_0^1 \exp\Big(-\frac{x^2+y^2}{4xy}u\Big)u^{\lambda+\frac12-\frac\alpha2}du\bigg],
\qquad x,y\in(0,\infty).
\end{align*}
According to \cite[Lemma 3.1]{NS}, 
$$\int_0^1\exp\Big(-\frac{(x-y)^2}{4xyu}\Big)u^{\frac{\alpha-4}2}du
   \lesssim \exp\Big(-c\frac{(x-y)^2}{xy}\Big)\Big(\frac{(x-y)^2}{xy}\Big)^{\frac{\alpha-2}2},
   \qquad x,y\in(0,\infty), x\ne y,$$
and
$$\int_0^1 \exp\Big(-\frac{x^2+y^2}{4xy}u\Big)u^{\lambda+\frac12-\frac\alpha2}du
\lesssim \Big(\frac{x^2+y^2}{xy}\Big)^{-\lambda-\frac32+\frac\alpha2},
   \qquad x,y\in(0,\infty), x\ne y.$$
It follows that for $x,y\in(0,\infty)$ with $x\ne y$,
\begin{align*}
&\Big|\frac\partial{\partial x} K^\alpha(x,y)\Big|\\
&\lesssim (x+y)\bigg[\exp\Big(-c\frac{(x-y)^2}{xy}\Big) |x-y|^{\alpha-2}(xy)^{-\lambda}\Big(\frac1{x+y}+\frac1{\sqrt{xy}}\Big)+ (x+y)^{\alpha-2\lambda-3}\bigg].
\end{align*}
Suppose that $(x-y)^2\le xy$. Then, $xy\sim (x+y)^2$ and we have
$$\Big|\frac\partial{\partial x} K^\alpha(x,y)\Big|\lesssim \frac{|x-y|^\alpha}{|x-y|^2(x+y)^{2\lambda}}.$$
On the other hand, if $(x-y)^2\ge xy$, then either $y<x/C$ or $y>Cx$, for some fixed $C>1$.
Suppose that $y<x/C$. We get
\begin{align*}
\Big|\frac\partial{\partial x} K^\alpha(x,y)\Big|
&\lesssim e^{-c\frac xy}|x-y|^{\alpha-2}(xy)^{-\lambda}\Big(1+\sqrt{\frac xy}\Big)+(x+y)^{\alpha-2\lambda-2}\\
&\lesssim |x-y|^{\alpha-2}x^{-2\lambda}+(x+y)^{\alpha-2\lambda-2}\\
&\lesssim \frac{|x-y|^\alpha}{|x-y|^2(x+y)^{2\lambda}}.
\end{align*}
When $y>Cx$ it is sufficient to exploit symmetry. Thus, \eqref{frac smooth 1} is proved.

Let $x,x',y\in\mathbb R_+$ with  $|x-x'|<|x-y|/2$. We have that
\begin{align*}
|K^\alpha(x,y)-K^\alpha(x',y)|
&\lesssim \bigg|\int_{x'}^x \frac{|z-y|^\alpha}{|z-y|^2(z+y)^{2\lambda}} dz\bigg|\\
&\lesssim |x-x'|\frac1{|x-y|^{2-\alpha}(x+y)^{2\lambda}}\\
&\lesssim   \frac{|x-x'|}{|x-y|} \frac{|x-y|^\alpha}{m_\lambda(I(x,|x-y|))}.
\end{align*}
The proof is complete.
\end{proof}

Next, we establish the following auxiliary maximal function estimate, which will be used in the proof of Theorem \ref{thm 1.7}.

Suppose for each interval $I$ we have a function $h^I$, defined on this interval. Let $\beta\ge 0$ and define
$$h_\beta(x)=\sup_{I\ni x}\frac1{m_\lambda(I)^{1+\frac \beta{\mathbf  Q}}}
                       \int_I |h^I(y)|dm_\lambda(y),\qquad x\in (0,\infty),$$
where the supremum is taken over all the intervals $I\subseteq (0,\infty)$ such that $x\in I$.
To prove Theorem  \ref{thm 1.7}, we need the following lemma.

\begin{lem}\label{pq lemma}
Let $0\le \beta<\alpha<\infty$ and $1<p<q<\infty, 1/p-1/q=(\alpha-\beta)/{\mathbf  Q}$. Suppose for each interval $I$ we have a function $h^I$, defined on this interval.
Then, 
\begin{equation}\label{lemma 5.1}
\|h_\beta\|_{L^q(\mathbb R_+, dm_\lambda)} \lesssim \|h_\alpha\|_{L^p(\mathbb R_+, dm_\lambda)},
\end{equation}
where the implied constant depends only on $p,q,\alpha$ and $\lambda$.
\end{lem}
}

\begin{proof}
Let $0<r<\frac{\mathbf  Q}{\alpha-\beta}$ and let $I$ be an interval in $(0,\infty)$. We have that
\begin{align*}
\frac1{m_\lambda(I)^{1+\frac \beta{\mathbf  Q}}}
                       \int_I |h^I|dm_\lambda
&\le m_\lambda(I)^{\frac{\alpha-\beta}{\mathbf  Q}}\int_{u\in I}h_\alpha(u) \\
&\le m_\lambda(I)^{\frac{\alpha-\beta}{\mathbf  Q}-\frac1r}
       \bigg(\int_I h_\alpha(u)^r dm_\lambda(u)\bigg)^{1/r} \\
&=\bigg(\frac1{m_\lambda(I)^{1-\frac{r(\alpha-\beta)}{\mathbf  Q}}}
         \int_I h_\alpha(u)^r dm_\lambda(u)\bigg)^{1/r} \\
&\lesssim \bigg(\int_I \frac{h_\alpha(u)^r}{m_\lambda(I(x,|x-u|))^{1-\frac{r(\alpha-\beta)}{\mathbf  Q}}}
         dm_\lambda(u)\bigg)^{1/r}, \quad x\in I. 
\end{align*}
We define, for every $0<\gamma<1$, the fractional integral $I^\gamma_\lambda$ by
$$I^\gamma_\lambda(f)(x)=\int_0^\infty \frac{f(u)}{m_\lambda(I(x,|x-u|))^{1-\gamma}}dm_\lambda(u).$$
Then, 
\begin{equation}\label{hbeta}
h_\beta(x)\lesssim \{I^\gamma_\lambda[h_\alpha]^r(x)\}^{1/r},\qquad x\in (0,\infty).
\end{equation}
It follows from \cite[Theorem 1]{Pan} that the operator $I^\gamma_\lambda$ maps $L^{\tilde p}$ boundedly into $L^{\tilde q}$ whenever $1<\tilde p<\tilde q$ and $\frac1{\tilde p}=\frac1{\tilde q}+\gamma$.
Let $\tilde p=p/r$ and $\tilde q=q/r$ with $r<p$ and $\gamma=\frac{r(\alpha-\beta)}{\mathbf  Q}<1$ as above. Then with $g=I^\gamma_\lambda[(h_\alpha)^r]$, we have from \eqref{hbeta}
\begin{align*}
\|h_\beta\|_{L^q(\mathbb R_+, dm_\lambda)}
&\lesssim \|g^{1/r}\|_{L^q(\mathbb R_+, dm_\lambda)}
   =\|g\|^{1/r}_{L^{\tilde q}(\mathbb R_+, dm_\lambda)}\\
&\lesssim \|(h_\alpha)^r\|_{L^{\tilde p}(\mathbb R_+, dm_\lambda)}\\
&=\|h_\alpha\|_{L^p(\mathbb R_+, dm_\lambda)},
\end{align*}
which is \eqref{lemma 5.1}.
\end{proof}  

Now we  show Theorem  \ref{thm 1.7}.

\begin{proof}[Proof of Theorem  \ref{thm 1.7}]
Fix $x\in \mathbb R_+, k\in \mathbb Z$. Set $I=I(x, 2^{-k})$ and $f\in L^p(\mathbb R_+,dm_\lambda)$. 
Let $f_1=f\chi_{(2I)}$ and $f_2=f-f_1$. Observe that $[b,\tbz ^{-\alpha/2}]f=[b-b_I,\tbz ^{-\alpha/2}]f$, and we have:
\begin{align*}
&\frac1{m_\lambda(I) 2^{-k\beta}}\int_I |[b,\tbz ^{-\alpha/2}]f(y)-([b,\tbz ^{-\alpha/2}]f)_I | dm_\lambda(y)  \\
&=\frac1{m_\lambda(I) 2^{-k\beta}}\int_I |[b-b_I,\tbz ^{-\alpha/2}]f(y)-([b-b_I,\tbz ^{-\alpha/2}]f)_I|dm_\lambda(y)  \\
&\le \frac2{m_\lambda(I) 2^{-k\beta}}\int_I |[b-b_I,\tbz ^{-\alpha/2}]f(y)-\tbz ^{-\alpha/2}((b-b_I)f_2)(x)|dm_\lambda(y)  \\
&\le \frac2{m_\lambda(I) 2^{-k\beta}}\int_I |(b(y)-b_I)\tbz ^{-\alpha/2}f(y)|dm_\lambda(y)\\
&\quad +
       \frac2{m_\lambda(I) 2^{-k\beta}}\int_I |\tbz ^{-\alpha/2}((b-b_I)f_1)(y)|dm_\lambda(y)  \\
&\qquad + \frac2{ 2^{-k\beta}} \sup_{y\in I} |\tbz ^{-\alpha/2}((b-b_I)f_2)(y)-\tbz ^{-\alpha/2}((b-b_I)f_2)(x)|.
\end{align*}
Next, use Proposition \ref{defference} to obtain
\begin{align*}
 &\|[b,\tbz ^{-\alpha/2}]f\|_{\dot F_{\lambda ,q}^{\beta,\infty}} \\
 &\lesssim \Big\|\sup_{I(\cdot,2^{-k})} \frac1{m_\lambda(I) 2^{-k\beta}}\int_I |(b(y)-b_I)\tbz ^{-\alpha/2}f(y)|dm_\lambda(y)\Big\|_{L^q(\mathbb R_+, dm_\lambda)}\\
 &\qquad + \Big\|\sup_{I(\cdot,2^{-k})} \frac1{m_\lambda(I) 2^{-k\beta}}\int_I |\tbz ^{-\alpha/2}((b-b_I)f_1)(y)|dm_\lambda(y)\Big\|_{L^q(\mathbb R_+, dm_\lambda)}\\
 &\qquad + \Big\|\sup_{I(\cdot,2^{-k})}\frac1{2^{-k\beta}} \sup_{y\in I} |\tbz ^{-\alpha/2}((b-b_I)f_2)(y)-\tbz ^{-\alpha/2}((b-b_I)f_2)(\cdot)|\Big\|_{L^q(\mathbb R_+, dm_\lambda)}\\
 &\lesssim\Big\|\sup_{I(\cdot,2^{-k})} \frac1{m_\lambda(I) 2^{-k\beta}}\int_I |(b(y)-b_I)\tbz ^{-\alpha/2}f(y)|dm_\lambda(y)\Big\|_{L^q(\mathbb R_+, dm_\lambda)}\\
&\qquad + \Big\|\sup_{I(\cdot,2^{-k})} \frac1{m_\lambda(I)^{1+\frac{\alpha}{\mathbf Q}}2^{-k\beta}}\int_I |\tbz ^{-\alpha/2}((b-b_I)f_1)(y)|dm_\lambda(y)\Big\|_{L^p(\mathbb R_+, dm_\lambda)}\\
 &\qquad + \Big\|\sup_{I(\cdot,2^{-k})}\frac1{2^{-k\beta} m_\lambda(I)^{\frac{\alpha}{\mathbf Q}}}
  \sup_{y\in I} |\tbz ^{-\alpha/2}((b-b_I)f_2)(y)-\tbz ^{-\alpha/2}((b-b_I)f_2)(\cdot)|\Big\|_{L^p(\mathbb R_+, dm_\lambda)}\\
 &=: I\!I\!I_1+ I\!I\!I_2 +I\!I\!I_3,
\end{align*}
where the last two inequalities follow from Lemma \ref{pq lemma}.

We estimate these terms individually. For $1<t<p$, H\"older's inequality gives
\begin{align*}
&\frac1{m_\lambda(I)2^{-k\beta}}\int_I |(b(y)-b_I)\tbz ^{-\alpha/2}f(y)|dm_\lambda(y)\\
&\le \frac1{ 2^{-k\beta}}\Big(\frac1{m_\lambda(I)} \int_I |b(y)-b_I|^{t} dm_\lambda(y)\Big)^{\frac1{t'}} 
           \Big(\frac1{m_\lambda(I)}\int_I |\tbz ^{-\alpha/2}f(y)|^t dm_\lambda(y)\Big)^{\frac 1t}\\ 
&\lesssim \|b\|_{\Lambda^\beta}M_t(\tbz ^{-\alpha/2}f)(x),
\end{align*}
where the last inequality follows from Theorem \ref{lip}. 
Thus,  \cite[Theorem 2.1]{NS} implies that
\begin{align*}
I\!I\!I_1&\lesssim\|b\|_{\Lambda^\beta}\|M_t(\tbz ^{-\alpha/2}f)\|_{L^q(\mathbb R_+, dm_\lambda)} \\
&\lesssim \|b\|_{\Lambda^\beta}\|\tbz ^{-\alpha/2}f\|_{L^p(\mathbb R_+, dm_\lambda)} 
\\
&\lesssim \|b\|_{\Lambda^\beta}\|f\|_{L^p(\mathbb R_+, dm_\lambda)} .
\end{align*}

To estimate $I\!I\!I_2$, we choose $s, 1<s<p$, and $\bar s$ such that $\frac 1s-\frac 1{\bar s}=\frac{\alpha}{\mathbf  Q}$. 
By H\"older's inequality,
\begin{align*}
 &\frac1{m_\lambda(I)^{1+\frac{\alpha}{\mathbf  Q}} 2^{-k\beta}}\int_I |\tbz ^{-\alpha/2}((b-b_I)f_1)(y)|dm_\lambda(y)\\
 &\le\frac1{m_\lambda(I)^{1+\frac{\alpha}{\mathbf  Q}}2^{-k\beta}}\|\tbz ^{-\alpha/2}((b-b_I)f_1)\|_{L^{\bar s}(\mathbb R_+, dm_\lambda)}\ m_\lambda(I)^{1-\frac1{\bar s}}\\
&=\frac1{m_\lambda(I)^{\frac 1s}2^{-k\beta}}\|\tbz ^{-\alpha/2}((b-b_I)f_1)\|_{L^{\bar s}(\mathbb R_+, dm_\lambda)}.
 \end{align*}
 Then we use  \cite[Theorem 2.1]{NS} to obtain 
$$
\frac1{m_\lambda(I)^{1+\frac{\alpha}{\mathbf  Q}}2^{-k\beta}}\int_B |\tbz ^{-\alpha/2}((b-b_I)f_1)(y)|dm_\lambda(y)\lesssim \frac1{m_\lambda(I)^{\frac1s}  2^{-k\beta}}\|(b-b_I)f_1\|_{L^s(\mathbb R_+, dm_\lambda)}.
$$
It is easy to check
\begin{align*}
&{1\over m_\lambda(I)}\int_{\mathbb R_+} |b(x')-b_I|^s |f_1(x')|^s dm_\lambda(x')\\
&\lesssim  {1\over m_\lambda(2I)}\int_{2I} |b(x')-b_I|^s |f(x')|^s dm_\lambda(x')\\
&\lesssim  \bigg(\sup_{y\in 2I } |b(y)-b_I|\bigg)^s {1\over {m_\lambda(2I)}}\int_{2I}  |f(x')|^s dm_\lambda(x')\\
&\lesssim  \bigg(\sup_{y\in 2I }{1\over m_\lambda(I)} \int_I |b(y)-b(z)|dm_\lambda(z) \bigg)^s M^s_s(f)(x)\\
&\lesssim  \bigg( 2^{-k\beta} \sup_{y\in 2I } {1\over m_\lambda(I)}\int_I \sup_{|y-z|\ne 0}{|b(y)-b(z)|\over |y-z|^\beta }dm_\lambda(z) \bigg)^s M^s_s(f)(x)\\
&\lesssim  \|b\|_{\Lambda^\beta}^s 2^{-ks\beta}M^s_s(f)(x).
\end{align*}
We see that 
\begin{align*}
&\|(b-b_I)f_1\|_{L^s(\mathbb R_+, dm_\lambda)}
\lesssim  \|b\|_{\Lambda^\beta}2^{-k\beta} m_\lambda(I)^{1\over s}M_s(f)(x).
\end{align*}

Thus, we have that 
$$\frac1{m_\lambda(I)^{1+\frac{\alpha}{\mathbf  Q}} 2^{-k\beta}}\int_I |\tbz ^{-\alpha/2}((b-b_I)f_1)(y)|dm_\lambda(y)\\
\lesssim \|b\|_{\Lambda^\beta}M_s(f)(x)
$$
and hence
$I\!I\!I_2\lesssim \|b\|_{\Lambda^\beta}\|M_s(f)\|_{L^p(\mathbb R_+, dm_\lambda)}\lesssim  \|b\|_{\Lambda^\beta}\|f\|_{L^p(\mathbb R_+, dm_\lambda)}.$

Finally, we estimate $I\!I\!I_3$.  By \eqref{frac smooth},
\begin{align*}
& \frac1{2^{-k\beta} m_\lambda(I)^{\frac \alpha{\mathbf  Q}}} \bigg|\tbz ^{-\alpha/2}((b-b_I)f_2)(y)-\tbz ^{-\alpha/2}((b-b_I)f_2)(x)\bigg| \\
&=  \frac1{2^{-k\beta} m_\lambda(I)^{\frac \alpha{\mathbf  Q}}}  \bigg|\int_{\mathbb R_+} [K^\alpha(y,z) -K^\alpha(x,z)]  (b(z)-b_I)f_2(z)dm_\lambda(z)\bigg| \\
&\lesssim  \frac1{2^{-k\beta} m_\lambda(I)^{\frac \alpha{\mathbf  Q}}} \int_{\mathbb R_+} \frac{|x-y|}{|x-z|} \frac{|x-z|^\alpha}{m_\lambda(I(x,|x-z|))}   |b(z)-b_I||f_2(z)|dm_\lambda(z)\\
&\lesssim  \frac1{2^{-k\beta} m_\lambda(I)^{\frac \alpha{\mathbf  Q}}}  \int_{(2I)^c}  \frac{|x-y|}{|x-z|} \frac{|x-z|^\alpha}{m_\lambda(I(x,|x-z|))}  |b(z)-b_I||f(z)|dm_\lambda(z).
\end{align*}
Since 
$$m_\lambda(I(x,r))\sim x^{2\lambda}r+ r^{\mathbf  Q}\ge r^{\mathbf  Q},$$
we obtain 
\begin{align*}
& \frac1{2^{-k\beta} m_\lambda(I)^{\frac \alpha{\mathbf  Q}}} \bigg|\tbz ^{-\alpha/2}((b-b_I)f_2)(y)-\tbz ^{-\alpha/2}((b-b_I)f_2)(x)\bigg| \\
&\lesssim \frac1{2^{-k\beta}}\sum_{j=1}^\infty \int_{2^{j-k}<|x-z|\le 2^{j-k+1}}  
          2^{-j(1-\alpha)}\frac1{m_\lambda(I(x,|x-z|))}|b(z)-b_I| |f(z)|dm_\lambda(z) \\
&\lesssim \frac1{2^{-k\beta}}\sum_{j=1}^\infty \frac1{m_\lambda(I(x,2^{j-k}))}
      \int_{|x-z|\le 2^{j+1-k}}  2^{-j(1-\alpha)}    \big(|b(z)-b_{2^jI}|+|b_{2^jI}-b_I|\big)|f(z)|dm_\lambda(z) \\
&\lesssim \frac1{2^{-k\beta}}\sum_{j=1}^\infty \frac{2^{-j(1-\alpha)}}{m_\lambda(I(x,2^{j-k}))}\int_{|x-z|\le 2^{j-k+1}} 
          |b(z)-b_{2^jI}||f(z)|dm_\lambda(z)\\
 &\qquad +\sum_{j=1}^\infty2^{-j(1-\alpha)}\|b\|_{\Lambda^\beta} 2^{(j-k)\beta} M(f)(x)j \\
 &\lesssim \Big(\|b\|_{\Lambda^\beta} M_{ts}(f)(x) + \|b\|_{\Lambda^\beta}M(f)(x)\Big) \sum_{j=1}^\infty 2^{-j(1-\alpha-\beta)}j\\
 &\lesssim \Big(\|b\|_{\Lambda^\beta} M_{ts}(f)(x) + \|b\|_{\Lambda^\beta}M(f)(x)\Big)\end{align*}
 for $\alpha+\beta<1$.
 
Hence, we have 
$$I\!I\!I_3\lesssim \|b\|_{\Lambda^\beta}  \|f\|_{L^p(\mathbb R_+, dm_\lambda)}$$
for $\alpha+\beta<1$.
Since $\tbz ^{-\alpha/2}\tbz ^{-\beta/2}=\tbz ^{-{(\alpha+\beta)}/2}$, the assumption $\alpha+\beta<1$ can be removed by the method in \cite[Facts 1.9 and 1.10 ]{P}.  The proof is complete.
\end{proof}

\bigskip

{\bf Acknowledgments:} JB's partially supported by grant PID2019-106093GB-I00 from the Spanish Government.
 XD and JL's research supported by ARC DP 220100285. 
MYL's research supported by MOST 110-2115-M-008-009-MY2.
BDW’s research is supported in part by National Science Foundation Grants DMS \#1800057, \#2054863, and \#20000510 and ARC DP 220100285.

\smallskip

Departamento de Análisis Matemático, Universidad de La Laguna, Campus de Anchieta, Avda. Astrofísico Francisco Sánchez, s/n, 38271, La Laguna (Sta. Cruz de Tenerife), Spain.

\smallskip

{\it E-mail}: \texttt{jbetanco@ull.es}
\vspace{0.3cm}



\smallskip

Department of Mathematics, Macquarie University, NSW, 2109, Australia.

\smallskip

{\it E-mail}: \texttt{xuan.duong@mq.edu.au}

\vspace{0.3cm}


\smallskip

Department of Mathematics, National Central University, Taiwan

\smallskip

{\it E-mail}: \texttt{mylee@math.ncu.edu.tw}

\vspace{0.3cm}



Department of Mathematics, Macquarie University, NSW, 2109, Australia.

\smallskip

{\it E-mail}: \texttt{ji.li@mq.edu.au}

\vspace{0.3cm}



Department of Mathematics, Washington University--St. Louis, St. Louis, MO 63130-4899 USA

\smallskip

{\it E-mail}: \texttt{wick@math.wustl.edu}

\vspace{0.3cm}

\end{document}